\documentclass[11pt, a4paper]{amsart}
\usepackage{amsmath, enumerate}
\usepackage{amsxtra}
\usepackage{amscd}
\usepackage{amsthm}
\usepackage{amsfonts}
\usepackage{amssymb}
\usepackage {pstricks}
\usepackage{pstricks,pst-node}
\usepackage[all]{xy}

\usepackage{srcltx}

\usepackage[UKenglish]{babel}

\newtheorem{theorem}{Theorem}[section]
\newtheorem*{theorem*}{Theorem}
\newtheorem{lemma}[theorem]{Lemma}
\newtheorem{proposition}[theorem]{Proposition}
\newtheorem*{proposition*}{Proposition}
\newtheorem{corollary}[theorem]{Corollary}
\newtheorem{definition}[theorem]{Definition}
\newtheorem{conjecture}[theorem]{Conjecture}

\theoremstyle{definition}
\newtheorem{remark}[theorem]{Remark}

\numberwithin{equation}{section}

\def\1{\hbox{1\kern-.35em\hbox{1}}}

\newcommand{\N}{{\mathbb N}}
\newcommand{\Z}{{\mathbb Z}}
\newcommand{\mZ}{{\mathbb Z}}

\newcommand{\R}{{\mathbb R}}
\newcommand{\C}{{\mathbb C}}
\newcommand{\mC}{{\mathbb C}}

\newcommand{\tto}{\twoheadrightarrow}

\newcommand{\fd}{\rm fin.dim}
\newcommand{\pd}{{\rm pd}}

\newcommand{\len}{\mathtt{l}}

\newcommand{\fb}{{\mathfrak b}}

\newcommand{\fg}{{\mathfrak g}}
\newcommand{\fh}{{\mathfrak h}}

\newcommand{\fs}{{\mathfrak s}}
\newcommand{\fl}{{\mathfrak l}}
\newcommand{\fn}{{\mathfrak n}}

\newcommand{\ad}{{\rm ad}}
\newcommand{\id}{{\rm id}}
\newcommand{\End}{{\rm End}}
\newcommand{\Hom}{{\rm Hom}}
\newcommand{\ind}{{\rm Ind}}
\newcommand{\coind}{{\rm Coind}}
\newcommand{\Ind}{{\rm Ind}^{\fg}_{\fg_{\oa}}}

\newcommand{\Ext}{{\rm Ext}}

\newcommand{\im}{{\rm im}}
\newcommand{\res}{{\rm Res}}

\newcommand{\Res}{{\rm Res}^{\fg}_{\fg_{\oa}}}

\newcommand{\cL}{{\mathcal L}}

\newcommand{\cA}{{\mathcal A}}

\newcommand{\cF}{{\mathcal F}}
\newcommand{\cS}{{\mathcal S}}

\newcommand{\cO}{{\mathcal O}}

\newcommand{\oa}{\bar{0}}
\newcommand{\ob}{\bar{1}}

\begin{document}

\title[Category~$\cO$ for $\mathfrak{gl}(m|n)$]{Homological invariants in category~$\cO$ for the general linear superalgebra}

\author{Kevin Coulembier}
\author{Vera Serganova}
\maketitle
{\center{\small{Department of Mathematical Analysis\\ Ghent University \\Krijgslaan 281, Gent, Belgium\\}}}

\vspace{0.5mm}

{\center{\small{Department of Mathematics\\ University of California at Berkeley \\Berkeley, CA 94720, USA\\}}}

\date{}

\vspace{-2mm}

\begin{abstract}
We study three related homological properties of modules in the BGG category~$\cO$ for basic classical Lie superalgebras, with specific focus on the general linear superalgebra. These are the projective dimension, associated variety and complexity. We demonstrate connections between projective dimension and singularity of modules and blocks. Similarly we investigate the connection between complexity and atypicality. This creates concrete tools to describe singularity and atypicality as homological, and hence categorical, properties of a block. However, we also demonstrate how two integral blocks in category~$\cO$ with identical global categorical characteristics of singularity and atypicality will generally still be inequivalent. This principle implies that category~$\cO$ for $\mathfrak{gl}(m|n)$ can contain infinitely many non-equivalent blocks, which we work out explicitly for $\mathfrak{gl}(3|1)$. All of this is in sharp contrast with category~$\cO$ for Lie algebras, but also with the category of finite dimensional modules for superalgebras. Furthermore we characterise modules with finite projective dimension to be those with trivial associated variety. We also study the associated variety of Verma modules. To do this, we also classify the orbits in the cone of self-commuting odd elements under the action of an even Borel subgroup.
\end{abstract}


\section{Introduction}

The Bernstein-Gelfand-Gelfand category~$\cO$ associated to a triangular decomposition of a finite dimensional contragredient Lie (super)algebra is an important and intensively studied object in modern representation theory, see e.g. \cite{MR2428237}. Category~$\cO$ for basic classical Lie superalgebras $\fg$ is not yet as well understood as for semisimple Lie algebras and exhibits many novel features. However, for the particular case of the general linear superalgebra $\mathfrak{gl}(m|n)$, category~$\cO$ has a Kazhdan-Lusztig (KL) type theory, introduced by Brundan in~\cite{Brundan} and proved to be correct by Cheng, Lam and Wang in~\cite{CLW}. This determines the characters of simple modules algorithmically. Moreover, in this case, category~$\cO$ is Koszul, as proved by Brundan, Losev and Webster in~\cite{Brundan3, BLW}. In the current paper, we note that this theory is also an `abstract KL theory' in the sense of~\cite{CPS1, CPS2}. Observe that also for $\mathfrak{osp}(2m+1|2n)$, a KL type theory has been introduced and established by Bao and Wang in~\cite{BW}.

Our main focus in the current paper is the study of three homological invariants in category~$\cO$, specifically for $\mathfrak{gl}(m|n)$, and their applications to the open question concerning the classification of non-equivalent blocks.

The {\em associated variety} of a module $M\in\cO$ is the set of self-commuting odd elements $x$ of the Lie superalgebras for which $M$ has non-trivial $\C[x]$-homology, see \cite{Duflo, Vera}. Since the associated variety of a module in $\cO$ consists of orbits of the even Borel subgroup, we classify such orbits. Then we investigate the associated variety of Verma modules for $\mathfrak{gl}(m|n)$ with distinguished Borel subalgebra, leading in particular to a complete description for the cases $\mathfrak{gl}(1|n)$, $\mathfrak{gl}(m|1)$ and $\mathfrak{gl}(2|2)$.

The {\em projective dimension} of an object in an abelian category with enough projective objects is the length of a minimal projective resolution. Contrary to category~$\cO$ for semisimple Lie algebras, category~$\cO$ for $\fg$ contains modules with infinite projective dimension. We obtain two characterisations for the abelian subcategory of modules having finite projective dimension. The first one (valid for $\mathfrak{gl}(m|n)$) is as the category of modules having trivial associated variety. The second one (valid for arbitrary~$\fg$) is as an abelian category generated by the modules induced from the underling Lie algebra. Then we determine the projective dimension of injective modules for $\mathfrak{gl}(m|n)$ and use this to obtain the finitistic global homological dimension of the blocks, which builds on and extends some results of Mazorchuk in~\cite{CouMaz2, MR2366357, preprint}. Concretely, we show that this global categorical invariant of the blocks is determined by the singularity of the core of the central character. The results also provide the means to describe the level of `dominance' of a simple module in terms of the projective dimension of its injective envelope. The relevance of this categorical description of the dominance lies in the fact that, contrary to the Lie algebra case, one cannot use the projective dimension of the simple module itself for this, as the latter dimension will be infinite as soon as the module is atypical.

To deal with modules with infinite projective dimension we define the {\em complexity} of a module as the polynomial growth rate of a minimal projective resolution. The concrete motivation is to obtain a tool to homologically and categorically describe atypicality, similar to the description of singularity and dominance by projective dimension in the previous paragraph. We prove that our notion of complexity is well defined on category~$\cO$ for any basic classical Lie superalgebra, meaning that complexity of all modules is finite. Then we study the complexity of Verma (for the distinguished Borel subalgebra) and simple modules for $\mathfrak{gl}(m|n)$ and the relation with the degree of atypicality. Similar results for the category of finite dimensional weight modules of $\mathfrak{gl}(m|n)$ have been obtained by Boe, Kujawa and Nakano in~\cite{BKN1, BKN2}.

Integral blocks in category~$\cO$ for a Lie algebra are equivalent if they have the same singularity, see \cite{SoergelD}. Similarly, the blocks of the category of finite dimensional modules of a basic classical Lie superalgebra depend (almost) only on the degree of atypicality, see \cite{MR2734963, Lilit}. The {\em classification of non-equivalent blocks} in category~$\cO$ for a Lie superalgebra is an open question. Our main result here is the fact that a combination of the two aforementioned global categorical characteristics does not suffice to separate between non-equivalent blocks.  More precisely, we use our results on projective dimensions to materialise subtle local differences in blocks with similar global properties into categorical invariants. We do this explicitly for regular atypical blocks for $\mathfrak{sl}(3|1)$, resulting in the fact that {\em all} such blocks are non-equivalent (while they have the same singularity and degree of atypicality). In particular this implies that, contrary to category~$\cO$ for Lie algebras and the category of finite dimensional weight modules for Lie superalgebras, category~$\cO$ for Lie superalgebras can contain infinitely many non-equivalent blocks. Furthermore the results demonstrate that, in general, equivalences between integral blocks will be very rare. This is summarised in Figure \ref{tabb}, where~$\cF$ represents the category of finite dimensional weight modules and $\fg_{\oa}$ a reductive Lie algebra.
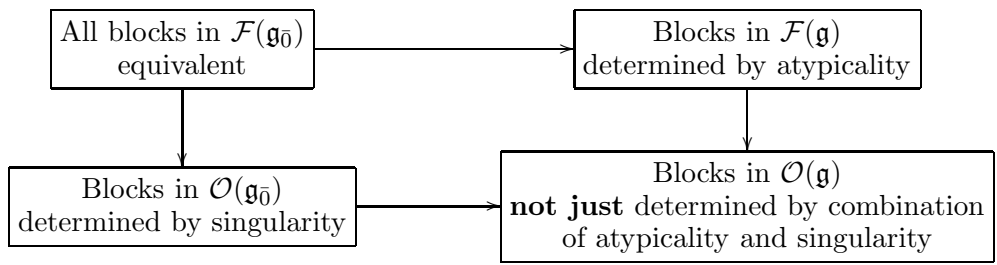
\begin{figure}\label{tabb}
\caption{Equivalence of blocks}
\[
\xymatrix{
 *+[F]\txt{All blocks in $\cF(\fg_{\oa})$ \\ equivalent}  \ar[d]\ar[rr]&& *+[F]\txt{Blocks in $\cF(\fg)$ \\ determined by atypicality}  \ar[d]\\
*+[F]\txt{Blocks in $\cO(\fg_{\oa})$\\determined by singularity}\ar[rr]&&*+[F]\txt{Blocks in $\cO(\fg)$\\ {\bf not just} determined by combination\\ of atypicality and singularity}
}
\]
\end{figure}

We hope that our results can be applied in the quest to obtain a full classification of non-equivalent blocks in category~$\cO$ for $\mathfrak{gl}(m|n)$. Furthermore, in further work we aim to strengthen the equivalence between trivial associative variety and zero complexity for a module to a more general link between complexity and the associated variety and in particular to determine the complexity of simple modules.

The paper is structured as follows. In Section~\ref{secpre} we recall some preliminary results. In Section~\ref{secext} we study extensions between Verma and simple modules, in connection with Brundan's Kazhdan-Lusztig theory. In Section~\ref{secfpd} we obtain the characterisations for the modules with finite projective dimension. In Section~\ref{secassvar} we study the self-commuting cone and associated varieties in relation with category~$\cO$. In Section~\ref{secblocks} we study projective dimensions and obtain the result on the non-equivalence of blocks. In Section~\ref{seccomp} we introduce and study the notion of complexity. Finally, in the appendices we illustrate certain results for the example of $\mathfrak{gl}(2|1)$ and carry out some technicalities.

\section{Preliminaries}
\label{secpre}

In this paper we work with basic classical Lie superalgebras. We refer to Chapters 1-4 in~\cite{bookMusson} for concrete definitions. We denote a basic classical Lie superalgebra by $\fg$ and its even and odd part by $\fg_{\oa}\oplus\fg_{\ob}=\fg$. A Borel subalgebra will be denoted by $\fb$, a Cartan subalgebra by $\fh$ (which is the same as a Cartan subalgebra of $\fg_{\oa}$) and the set of positive roots by $\Delta^+$. The even and odd positive roots are then denoted by $\Delta^+_{\oa}$ and $\Delta^+_{\ob}$. The set of integral weights is denoted by $P_0\subset\fh^\ast$. The Weyl group $W=W(\fg:\fh)$ is the same as the Weyl group $W(\fg_{\oa}:\fh)$. We fix a $W$-invariant form $(\cdot,\cdot)$ on $\fh^\ast$ as in Theorem 5.4.1 of~\cite{bookMusson}. We define $\rho=\frac{1}{2}(\sum_{\alpha\in\Delta^+_{\oa}}\alpha)-\frac{1}{2}(\sum_{\gamma\in\Delta_{\ob}^+}\gamma)$. The dot action is then given by
$$w\cdot\lambda=w(\lambda+\rho)-\rho.$$

\subsection{Basic classical Lie superalgebras of type~$A$}

Mostly $\fg$ will be one of the following
\begin{equation}\label{listA}\fg=\begin{cases}\mathfrak{gl}(m|n)\\ \mathfrak{sl}(m|n)\,\mbox{ with }\, m\not=n\\\mathfrak{pgl}(n|n)\simeq \mathfrak{gl}(n|n)/\mathfrak{z}(\mathfrak{gl}(n|n)).\end{cases}\end{equation}
Then we use the standard $\mZ$-grading $\fg=\fg_{-1}\oplus\fg_0\oplus\fg_1,$
with $\fg_{\oa}=\fg_0$ and $\fg_{\ob}=\fg_{-1}\oplus \fg_{1}$. We fix an element in the centre of $\fg_0$ which we denote by $z\in\mathfrak{z}(\fg_0)$ and which satisfies
\begin{equation}\label{gradel}
[z,X]=X,\quad\forall \,X\in\fg_1\quad\mbox{and}\qquad [z,Y]=-Y,\quad\forall\, Y\in \fg_{-1}. 
\end{equation}
The necessity of such an element is the reason we can not always include the other classical Lie superalgebras of type~$A$, {\it viz.} $\mathfrak{sl}(n|n)$ and $\mathfrak{psl}(n|n)$. Note that technically $\mathfrak{sl}(n|n)$ and $\mathfrak{pgl}(n|n)$ are not basic. In Remark \ref{error} we provide an example of properties concerning the associated variety that fail for $\mathfrak{sl}(n|n)$.

We choose the standard basis $\{\varepsilon_i,i=1,\cdots,m\}\cup \{\delta_j,j=1,\cdots,n\}$ of $\fh^\ast$ for $\mathfrak{gl}(m|n)$. Unless stated otherwise, the positive roots of any algebra in~\eqref{listA} are chosen to be
\begin{equation}\label{roots}
 \Delta^+=\left\{
\begin{array}{ll}
\varepsilon_{i} - \varepsilon_{j} &\hbox{for $1\le i< j\le m$,}\\
\varepsilon_{i} -\delta_j &\hbox{for
$1\le i\le m$
 and  $1\le  j\le n$,}\\
\delta_{i} - \delta_{j} &\hbox{for $1\le i< j\le n$.}
\end{array}
\right.
\end{equation}
This choice corresponds to the so-called distinguished system of positive roots, or the distinguished Borel subalgebra.

The Lie superalgebra generated by the positive root vectors is denoted by $\fn$, so we have $\fb=\fh\oplus\fn$, $\fg_1=\fn\cap \fg_{\ob}$ and set $\fn_0:=\fn\cap\fg_0$.
We define $\rho_0=\frac{1}{2}(\sum_{\alpha\in\Delta^+_0}\alpha)$. Note that $\rho_1=\rho-\rho_0$ is orthogonal to all even roots, so $w\cdot\lambda=w(\lambda+\rho_0)-\rho_0$ for $\lambda\in\fh^\ast$ and $w\in W$. 

For the remainder of this subsection we restrict to $\fg=\mathfrak{gl}(m|n)$. We will use
\[\delta=-\sum_{i=1}^m i\varepsilon_{i} +\sum_{i=1}^n (m-i+1)\delta_i\;\,\in \fh^\ast, \]
rather than $\rho$, because the coefficients of $\delta$ are integers. The difference $\rho-\delta$ is orthogonal to all roots, so
$$w\cdot\lambda=w(\lambda+\delta)-\delta\qquad\mbox{and}\qquad (\lambda+\rho,\gamma)=(\lambda+\delta,\gamma),$$
for $w\in W$ and $\gamma\in \Delta_{\ob}^+$. We fix a bijection between integral weights $P_0\subset\fh^\ast$ and $\mZ^{m|n}$, by
\begin{equation} \label{hat}P_0\leftrightarrow \mZ^{m|n},\quad\lambda\mapsto \mu^\lambda\quad \mbox{with}\quad \mu^\lambda_i=(\lambda+\delta,\varepsilon_i)\quad\mbox{and}\quad \mu^\lambda_{m+j}=(\lambda+\delta,\delta_j).\end{equation}
As in~\cite{BLW} we use the notation $\Lambda:=\Z^{m|n}$. Note that, by the above $\mu^{w\cdot\lambda}=w(\mu^\lambda)$, where the latter term uses the regular action of  $W\simeq S_m\times S_n$ on $\Lambda=\Z^{m|n}$.

For any $\mu=(\alpha_1,\alpha_2,\cdots\alpha_m|\beta_1,\cdots\beta_n)\in \Lambda$, we refer to the integers $\alpha_1,\cdots\alpha_m$ as the labels of $\mu$ on the left side and to the integers $\beta_1,\cdots\beta_n$ as the labels of~$\mu$ on the right side.

The set of integral regular dominant weights is then described by
$$P_0^{++}=\{\lambda\in P_0\,|\, w\cdot\lambda <\lambda,\,\,\forall\, w\in W \},$$
where~$\mu\leq\lambda$ if $\lambda-\mu$ is a sum of positive roots.
We denote by $\Lambda^{++}$ the corresponding subset in $\Lambda$.
This set has the structure of a poset, which describes the highest weight structure 
of the category of finite dimensional weight modules $\cF$, see e.g. \cite{Brundan, MR1443186}.

Another convenient choice of Borel subalgebra is the anti-distinguished Borel subalgebra $\fb=\fb_{\oa}\oplus \fg_{-1}$. Note that the anti-distinguished Borel subalgebra of $\mathfrak{gl}(m|n)$ is mapped to the distinguished Borel subalgebra of $\mathfrak{gl}(n|m)$ under the isomorphism $\mathfrak{gl}(m|n)\simeq \mathfrak{gl}(n|m)$.

\subsection{BGG category~$\cO$}
\label{prelO}

Category~$s\cO$ for a basic classical Lie superalgebra $\fg$ with Borel subalgebra $\fb$ is defined as the full subcategory of all $\fg$-modules, where the objects are finitely generated; $\fh$-semisimple and locally $U(\fb)$-finite. Note that this definition does not depend on the actual choice of Borel subalgebra, only the even part $\fb_{\oa}:=\fb\cap\fg_{\oa}.$ For each Borel subalgebra $\fb$ (with $\fb\cap \fg_{\oa}=\fb_{\oa}$), this category has a (different) structure of a highest weight category. An alternative definition of $s\cO$ is as the full subcategory of all $\fg$-modules, where the objects $M$ satisfy~$\Res M\in\cO^0$, after neglecting parity, where~$\cO^0$ is the corresponding category for $\fg_{\oa}$. In what follows we will work with the full Serre subcategory~$\cO$ of $s\cO$, defined similarly as in Section 2 of \cite{Brundan3}. This means that to each weight one attaches a parity (in a consistent way) and only modules in which the corresponding weight spaces appear in said parity are allowed. Then one has $s\cO\simeq \cO\oplus\cO$, abstractly as categories, see Lemma 2.2 of \cite{Brundan3}. We denote the Serre subcategory of $\cO$ generated by modules admitting the central character $\chi$ by $\cO_{\chi}$. We denote the central character corresponding to $L(\lambda)$ by~$\chi_\lambda$.

Now we turn to this category for $\mathfrak{gl}(m|n)$. For an overview of the current knowledge we refer to the survey \cite{Brundan3}. The indecomposable blocks of category~$\cO$ for $\mathfrak{gl}(m|n)$ have been determined by Cheng, Mazorchuk and Wang in Theorem 3.12 of \cite{CMW}. They also proved, in Theorem 3.10 of 
\cite{CMW}, that every non-integral block is equivalent to an integral block in the category $\cO$ for a direct sum of several general linear superalgebras. 
Therefore, we can restrict to integral blocks for several of our purposes. 

The Serre subcategory of $\cO$ of modules with integral weight spaces is denoted by $\cO_{\mZ}$. The blocks of category~$\cO_\mZ$ can be described by linkage classes. The linkage class $\xi$ generated by $\mu\in P_0$ is
$$\xi=[\mu]=\{\lambda\in P_0\,|\,\chi_{\mu}=\chi_{\lambda}\}.$$
The indecomposable block $\cO_\xi$ is then defined as the full Serre subcategory of $\cO$ generated by the set of simple modules $\{L(\lambda)\,|\,\lambda\in\xi\}.$ The degree of atypicality of the weights in the linkage class is denoted by $\sharp \xi$. Furthermore we denote the central character corresponding to the block by $\chi_\xi$. The  Bruhat order $\preceq$ of \cite{Brundan, BLW, CMW} on $P_0$ is the minimal partial order satisfying
\begin{itemize}
\item if $s\cdot\lambda \le \lambda $ for a reflection $s\in W$ and $\lambda\in P_0$, we have $s\cdot\lambda \preceq \lambda $;
\item if $(\lambda+\rho,\gamma)=0$ for $\lambda\in P_0$ and $\gamma\in \Delta^+_1$, we have $\lambda-\gamma\preceq \lambda$.
\end{itemize}
The set $\Lambda\simeq P_0$ equipped with this Bruhat order is the poset for $\cO_{\Z}$ as a highest weight category. 
Note that the connected components of the Bruhat order are precisely the linkage classes.

As in~\cite{MR2734963} we define the core of $\chi_\xi$, which we denote by $\chi'_\xi$. This is the typical central character of $\mathfrak{gl}(m-\sharp\xi|n-\sharp\xi)$, corresponding to a weight in $\xi$ where~$\sharp\xi$ labels on each side are removed in order to create a typical weight. We also fix $w_0^\xi\in W$ as the longest element in the subgroup of the Weyl group of $\mathfrak{gl}(m-k|n-k)$ which stabilises a dominant weight corresponding to~$\chi'_\xi$.

We will use the translation functors on $\cO$ introduced in~\cite{Brundan} and studied further in~\cite{BLW, Ku}. Denote by $U=\mC^{m|n}$ the tautological module and let $F$ (resp. $E$) be the exact endofunctor of $\cO_{\mZ}$ defined by tensoring with $U$ (resp. $U^\ast$). Then $\{F_i\,|\,i\in\mZ\}$ and $\{E_i\,|\,i\in\mZ\}$ are the subfunctors of $F$ and $E$ corresponding to projection on certain blocks. According to Theorem 3.10 of~\cite{BLW}, this defines an $\mathfrak{sl}(\infty)$ tensor product categorification on $\cO_{\mZ}$.

Finally we introduce notation for some structural modules in category~$\cO$. For each~$\lambda\in\fh^\ast$ we denote the Verma module (the $\fg$-module induced from the one dimensional $\fb$-module on which $\fh$ acts through $\lambda$) by $M(\lambda)$. Its simple top is denoted by $L(\lambda)$. The indecomposable projective cover of $L(\lambda)$ is denoted by $P(\lambda)$ and the indecomposable injective hull by~$I(\lambda)$. The corresponding modules in $\cO^0$ are denoted by $L_0(\lambda),M_0(\lambda),P_0(\lambda), I_0(\lambda)$. We denote the Kac modules by $K(\lambda)=U(\fg)\otimes_{U(\fg_0\oplus\fg_1)}L_0(\lambda)$ and the dual Kac modules by $\overline{K}(\lambda)=U(\fg)\otimes_{U(\fg_0\oplus\fg_{-1})}L_0(\lambda)$. These modules are also co-induced, e.g.
\begin{equation}\label{indcoind}U(\fg)\otimes_{U(\fg_0\oplus\fg_{1})}L_0(\lambda)\simeq\Hom_{U(\fg_0+\fg_1)}(U(\fg),L_0(\lambda-2\rho_1)).\end{equation}

Note that for structural modules we will sometimes write $L(\mu^\lambda)$ to denote $L(\lambda)$ when~$\lambda\in P_0$, with slight abuse of notation.

By Theorem 25$(i)$ and Corollary 14 of~\cite{CouMaz2}, for any $\lambda\in\fh^\ast$ and $N\in\cO$ we have
\begin{equation}
\label{VerH}
\Ext^i_{\cO}(M(\lambda),N)\simeq \Hom_{\fh}(\C_{\lambda},H^i(\fn,N)),
\end{equation}
with $H^i(\fn,-)\simeq \Ext^i_{\fn}(\C,-)$, the Lie superalgebra cohomology of $\fn$.

\subsection{Brundan-Kazhdan-Lusztig theory} We review a few items of~\cite{Brundan, BLW, CLW}, to which we refer for details, see also the survey \cite{Brundan3}.

Let $\mathbb V$ be the natural $\mathfrak{sl}(\infty)$ module and $\mathbb W$ its restricted dual. The Lie algebra $\mathfrak{sl}(\infty)$ is generated by the  Chevalley generators $\{e_i,f_i\,|\,i\in\mZ\}$. Then $\Lambda=\mZ^{m|n}$ naturally parametrises a monomial basis of $\mathbb V^{\otimes m}\otimes \mathbb W^{\otimes n}$.  We denote such a monomial basis by $\{v_\lambda \,|\,\lambda\in \Lambda\}$. Identifying $[M(\lambda)]\in K(\cO_{\mZ})$ with $v_\lambda$ then leads to an isomorphism of vector spaces \begin{equation}K(\cO^{\Delta}_{\mZ})\leftrightarrow \mathbb V^{\otimes m}\otimes \mathbb W^{\otimes n},\label{GrothTens}
\end{equation}
with $K(\cO^{\Delta}_{\mZ})$ the Grothendieck group of the exact full subcategory~$\cO^{\Delta}_{\mZ}$ of $\cO_{\mZ}$, whose objects are
the modules admitting a Verma flag. It follows immediately from the standard filtration of $M(\lambda)\otimes U$ that we can define an $\mathfrak{sl}(\infty)$-action on $K(\cO_{\mZ})$ by $e_i [M]:= [E_i M]$ and $f_i[M]=[F_i M]$. The isomorphism \eqref{GrothTens} then becomes an $\mathfrak{sl}(\infty)$-module isomorphism.

For the quantised enveloping algebra $U_q(\mathfrak{sl}(\infty))$ we denote the corresponding module by $\dot{\mathbb V}^{\otimes m}\otimes \dot{\mathbb W}^{\otimes n}$. It turns out that $\dot{\mathbb V}^{\otimes m}\otimes \dot{\mathbb W}^{\otimes n}$ admits the Lusztig canonical basis. This basis is denoted by $\{\dot{b}_\mu,\, \mu\in \Lambda\}$ and the monomial basis by $\{\dot{v}_\lambda,\, \lambda\in\Lambda\}$. Then the polynomials $d_{\mu,\lambda}(q)$ and the inverse matrix, the KL polynomials $p_{\lambda,\nu}(-q)$, are defined as
$$\dot{b}_\mu=\sum_{\lambda\in\Lambda}d_{\mu,\lambda}(q)\dot{v}_{\lambda}\quad\mbox{and}\quad\dot{v}_{\lambda}=\sum_{\nu\in\Lambda}p_{\lambda,\nu}(-q)\dot{b}_\nu.$$
In \cite{Brundan}, Brundan conjectured and in~\cite{CLW} Cheng, Lam and Wang proved that
\begin{equation}\label{resCheng}(P(\mu): M(\lambda))=d_{\mu,\lambda}(1)=[M(\lambda):L(\mu)].\end{equation}
Furthermore Brundan, Losev and Webster proved in Theorem A of~\cite{BLW} that $\cO$ (has a graded lift which) is standard Koszul. This implies that the KL polynomials can be interpreted as a minimal projective resolution of the Verma module, or
\begin{equation}\label{defp}p_{\lambda,\nu}(q)=\sum_{k\ge 0}q^k\dim\Ext^k_{\cO}(M(\lambda), L(\nu)),\end{equation}
see also Section 5.9 of~\cite{BLW}. For $q=-1$, equation~\eqref{defp} is a direct consequence of equation~\eqref{resCheng} and the Euler-Poincar\'e principle.

Take an interval $I\subset \mZ$. Let $\mathfrak{sl}(I)\simeq \mathfrak{sl}(|I|+1)$ denote the Lie subalgebra of $\mathfrak{sl}(\infty)$ generated by the 
Chevalley generators $\{e_i,f_i\,|\,i\in I\}$. We set $I_+:=I\cup (I+1)$. Then $\Lambda_I$ is the sub-poset of $\Lambda$, consisiting of all vectors in 
$\mathbb Z^{m|n}$ with labels in the interval $I_+$. It is clear that $\Lambda_I$ is in bijection with the monomial basis of the $\mathfrak{sl}(I)$-module
${\mathbb V_I}^{\otimes m}\otimes {\mathbb W_I}^{\otimes n}.$
Since $\mathbb W_I$ is isomorphic to $\Lambda^{|I|}\mathbb V_I$ as a $\mathfrak{sl}(I)$-module, $\Lambda_I$ also corresponds to a poset of another highest weight 
category. 
The relevant category is a subcategory of the parabolic category~$\cO$ for $\mathfrak{gl}(m+|I|n)$, where the parabolic subalgebra has Levi part
$\mathfrak{gl}(1)^{\oplus m}\oplus\mathfrak{gl}(|I|)^{\oplus n}$, generated by simple modules with suitable restrictions on highest weights, 
see Definition~3.13 in~\cite{LW}. 
We denote this category by $\cO_I'$ and the bijection of the weights in $\Lambda_I$ with the set $\Lambda_I'$ of highest weights of the simple modules in 
$\cO'_I$ by $\phi_I:\Lambda_I\to\Lambda_{I}'$.

In Section 2.8 in~\cite{BLW}, two ideals $\Lambda_{\le I}$ and $\Lambda_{< I}$ in the poset are constructed with the property 
$\Lambda_I=\Lambda_{\le I}\backslash \Lambda_{< I}$. Then we have Serre (highest weight) subcategories $\cO_{< I}$ and $\cO_{\le I}$ in 
$\cO_{\mZ}$ and the quotient category~$\cO_I=\cO_{\le I}/\cO_{< I}$. This is a highest weight category with poset $\Lambda_I$. By the general theory of such 
subquotients of highest weight categories, it follows that 
\begin{equation}\label{extsubquo}
\Ext^\bullet_{\cO}(M(\mu), L(\lambda))\simeq \Ext^\bullet_{\cO_I}(M(\mu), L(\lambda))
\end{equation} if $\lambda,\mu\in\Lambda_I$, see Section 2.5 in~\cite{BLW}.

 It is proved in~\cite{BLW} via uniqueness of tensor product categorifications in~\cite{LW}, that the categories $\cO_I$ and $\cO_I'$ are equivalent.

\section{Extensions and Kazhdan-Lusztig theory}
\label{secext}
In this entire section we consider $\fg=\mathfrak{gl}(m|n)$ with, unless specified otherwise, distinguished Borel subalgebra. The results extend to $\mathfrak{sl}(m|n)$ if $m\not=n$.

\subsection{Length function and abstract Kazhdan-Lusztig theories}
\label{seclen}

We define a length function $l:\Lambda\times \Lambda \to\N$ on a poset $(\Lambda,\preceq)$ to be a function with domain $\{(\lambda,\mu)\,|\,\mu \preceq \lambda\}$ which satisfies $l(\lambda,\mu)=l(\lambda,\kappa)+l(\kappa,\mu)$ if $\mu\preceq\kappa\preceq\lambda$, with $l(\lambda,\mu)=0$ if and only if $\lambda=\mu$. Note that, in principal, a length function should be a function $l':\Lambda\to\Z$ such that $l(\lambda,\mu):=l'(\lambda)-l'(\mu)$ satisfies the above properties. However, in our case, it is possible to construct such an $l'$ from our $l$ by the procedure in Section 3-g in~\cite{Brundan}. As we will only need the difference in length between two comparable weights, we ignore this technicality.

Before going to $\mathfrak{gl}(m|n)$, we review the length function for any integral (possibly singular) block in (possibly) parabolic category~$\cO$ for some $\mathfrak{gl}(d)$ with $d\in\N$. In this case we can assume $\lambda$ and $\mu$ to belong to the same (integral) Weyl group orbit. For an integral block of category~$\cO$ we set $l(\lambda,\mu)=l(\lambda)-l(\mu)$, with $l(\lambda),l(\mu)$ as defined in Theorem 3.8.1 in~\cite{CPS1}. There is a unique dominant element in the orbit of $\lambda$ and $\mu$, which we denote by $\kappa$. Then there are unique elements of the Weyl group $w_1,w_2\in W$ such that these are the longest element satisfying $\mu=w_1\cdot\kappa$ and $\lambda=w_2\cdot\kappa$. Then we have $l(\lambda,\mu)=l(w_1)-l(w_2)$, where~$l:W\to\N$ is the length function on $W$ as a Coxeter group. For an integral block in parabolic category~$\cO$, we just have the same length function as for the corresponding block in the original category~$\cO$, but restricted to the poset of weights dominant for the Levi subalgebra. Also the Bruhat order is the restriction of the Bruhat order in the non-parabolic case.

Now we can define a length function for category~$\cO_{\mZ}$ for $\mathfrak{gl}(m|n)$ and distinguished Borel subalgebra. For the cases $\fg=\mathfrak{gl}(2|1)$ and $\fg=\mathfrak{gl}(1|2)$, this will be made explicit in Appendix \ref{gl21}.
\begin{lemma}\label{defL}
For any two $\lambda,\mu\in \Lambda$ with $\mu  \preceq\lambda $ and any interval $I$ such that $\lambda,\mu\in\Lambda_I$, set $l_I(\lambda,\mu)=l(\phi_I(\lambda),\phi_I(\mu))$. 

This value $l_I(\lambda,\mu)$ does not depend on the particular interval $I$. This leads to a well-defined length function $\mathtt{l}$ on $\Lambda\times \Lambda$ 
defined by
$$\len(\lambda,\mu)=l_I(\lambda,\mu)\qquad\mbox{for any $I$ such that }\lambda,\mu\in \Lambda_I.$$
\end{lemma}
\begin{proof}
Take $\mu=(\alpha_1,\cdots,\alpha_m|\beta_1,\cdots,\beta_n)\in \Lambda$. We choose some $b\in\mZ$ greater than the maximal value of the labels of $\mu$ and an 
$a\in\mZ$ smaller than the minimum of the labels of~$\mu$. Take the interval $I=[a,b]$. By the description in Section 2.2 in~\cite{BLW} and Section~3.5 
in~\cite{LW} we have
\begin{eqnarray*}&&\phi_{[a,b]}(\alpha_1,\cdots,\alpha_m|\beta_1,\cdots,\beta_n)\\
&=&(\alpha_1,\cdots,\alpha_m,b+1,b,\cdots^{\hat{\beta_1}},a,b+1,b,\cdots^{\hat{\beta_2}},a,\cdots,b+1,b,\cdots^{\hat{\beta_n}},a),
\end{eqnarray*}
where~$\hat{c}$ implies that the value $c$ is left out in the sequence of numbers which otherwise descend by $1$ in each step.

We have to prove that $l_{[a,b]}(\lambda,\mu)=l_{[a',b']}(\lambda,\mu)$ for $\lambda,\mu\in \Lambda_{[a,b]}$ and any $a'\le a$ and $b'\ge b$. By construction, $\phi_{[a,b]}(\lambda)$ and $\phi_{[a,b]}(\mu)$ are in the same orbit of $S_{m+n(b-a+1)}$. Denote the unique dominant weight in the orbit by $D_{a,b}$. Assume now that $w_\mu^{a,b}$ is the longest element in $S_{m+n(b-a+1)}$ such that 
$$\phi_{[a,b]}(\mu)= w_\mu^{a,b} D_{a,b}.$$
We now embed $S_{m+n(b-a+1)}$ into $S_{m+n(b-a+2)}$ by identifying it with $S_{m+n(b-a+1)}\times S_1^n$. It follows that  
$$w_{\mu}^{a-1,b} = x \, w_\mu^{a,b} y\, \quad \in S_{m+n(b-a+2)},$$
where~$y$ is the longest element of the subgroup $S_1^{m+n(b-a+1)}\times S_n$ and $x$ is given by
$$(s_{m+1}s_{m+2}\cdots s_{m+nd})(s_{m+d+2}s_{m+n+2}\cdots s_{m+nd+1})\cdots (s_{m+(n-1)(d+1)+1}\cdots s_{p-1}),$$
where we set $d=b-a+1$ and $p=m+n(b-a+2)=m+nd+n$. As $x$ and $y$ clearly do not depend on $\mu$, we find that indeed 
$$l(w_\mu^{a,b})-l(w_\lambda^{a,b})=l(w_\mu^{a-1,b})-l(w_\lambda^{a-1,b}).$$
A similar reasoning for $b\mapsto b+1$ concludes the proof.
\end{proof}

This definition has the following immediate consequence.
\begin{theorem}
\label{KLtheory}
For $\lambda,\mu\in\Lambda$, we have
$$\Ext^i_{\cO}(M(\mu),L(\lambda))\not=0\quad\Rightarrow \quad\mu\preceq\lambda\,\mbox{ with }\,i\le \len (\lambda,\mu)\,\mbox{ and }\,i\equiv \len (\lambda,\mu)\,({\rm mod} 2).$$
\end{theorem}
\begin{proof}
Through equations \eqref{defp}, \eqref{extsubquo} and the fact that KL polynomials of $\cO_I$ correspond to those in $\cO_I'$, this can be reduced to the corresponding statement on extensions between standard and simple modules in singular blocks of parabolic category~$\cO$ for Lie algebras, by Lemma~\ref{defL}.

This result is known in these categories. We give a sketch of a proof for completeness. For blocks in non-parabolic category~$\cO$ this is Theorem~3.8.1 in~\cite{CPS1}. To prove it in parabolic category~$\cO$, we consider the Koszul dual statement, concerning the radical filtration of standard modules, see \cite{Backelin}. Then, by the above, the statement is correct for regular blocks in parabolic category~$\cO$ for Lie algebras.
 The full (dual) result follows immediately from graded translation to the wall, see \cite{Stroppel}.
\end{proof}

By applying the work of Cline, Parshall and Scott this leads to the following corollary.
\begin{corollary}
\label{corsim}
The highest weight category~$\cO_{\mZ}$ with length function $\len$ has an abstract Kazhdan-Lusztig theory, according to Definition 2.1 in~\cite{CPS2}. Consequently, we have
$$\sum_{k\ge 0} q^k \dim\Ext_{\cO}^k(L(\lambda),L(\mu))=\sum_{\nu\in\Lambda} p_{\nu,\lambda}(q)p_{\nu,\mu}(q).$$
\end{corollary}
\begin{proof}
This is a special case of Corollary 3.9 in~\cite{CPS2}, using the duality on $\cO$.
\end{proof}

Comparison with equation~\eqref{defp} yields
\begin{equation}
\label{extsimple}
\dim \Ext^j_{\cO}(L(\lambda),L(\mu))=\sum_{i=0}^j\sum_{\nu\in\Lambda}\dim\Ext^i_{\cO}(M(\nu),L(\lambda))\,\dim\Ext_{\cO}^{j-i}(M(\nu),L(\mu)).
\end{equation}
This formula follows also by standard methods from the observation that category~$\cO$ for~$\mathfrak{gl}(m|n)$ is standard Koszul, see \cite{BLW}, the fact that there are no extensions from Verma modules to dual Verma modules and use of the duality functor. 

\begin{remark}
The analogue of equation~\eqref{extsubquo} does not hold for
$$\Ext^\bullet_{\cO}(L(\mu), L(\lambda))\leftrightarrow \Ext^\bullet_{\cO_I}(L(\mu), L(\lambda)),$$
which is confirmed by equation~\eqref{extsimple}, as the summation over $\nu$ goes out of $\Lambda_I$. However, using the subsequent Lemma~\ref{vancentre} it is possible to show that for each~$\lambda,\mu\in \Lambda_I$ and a fixed degree $j$, there is an interval $J_{\lambda,\mu,j}$, such that
$$\dim\Ext^j_{\cO}(L(\mu), L(\lambda))\simeq\dim \Ext^j_{\cO_{J_{\lambda,\mu,j}}}(L(\mu), L(\lambda)).$$
\end{remark}

\begin{remark}
\label{lengthnotsame}
\begin{enumerate}
\item The length function $\mathtt{l}$ of Lemma~\ref{defL} does not reduce to the length function for $\mathfrak{gl}(m)\oplus\mathfrak{gl}(n)$, when restricted to one Weyl group orbit.
\item The restriction of $\mathtt{l}$ to $\Lambda^{++}$ does not correspond to the known length function, as defined by Brundan in Section $3$-$g.$ of~\cite{Brundan}. It is impossible to find a length function on $\Lambda$ with such a restriction.
\end{enumerate}
Both properties are illustrated in Appendix \ref{gl21}.
\end{remark}

The Brundan KL theory for $\cF$ is also an abstract KL theory, as proved in Theorem~4.51(i) of \cite{Brundan}. There are connections between the KL polynomials for both categories.
\begin{lemma}
\label{lemOF}
The restriction of the polynomials $p_{\lambda,\nu}(q)$ in equation~\eqref{defp} to $\lambda,\nu\in \Lambda^{++}$ gives the KL polynomials of~\cite{Brundan, MR1443186} for the category~$\cF$. Namely,
$$\dim\Ext_{\cO}^j(M(\lambda),L(\nu))=\dim\Ext^j_{\cF}(K(\lambda),L(\nu)).$$
\end{lemma}

In the proof we will use the Hochschild-Serre (HS) spectral sequences, see e.g. Section 16.6 in~\cite{bookMusson}. Concretely we will apply the HS 
spectral sequence for the ideal $\fg_1\subset \fn$:
\begin{equation} 
\label{HSss}
H^p(\fn_0,H^q(\fg_1,L(\mu)))\,\Rightarrow\, H^{p+q}(\fn,L(\mu)).
\end{equation}

\begin{proof}
Equation~\eqref{VerH} allows us to use the spectral sequence \eqref{HSss}. The fact that $H^q(\fg_1,L(\nu))$ is a finite dimensional $\fg_0$-module and Kostant 
cohomology (see Theorem 5.14 of \cite{Kostant}) imply that for any $\lambda\in\Lambda^{++}$
$$\Hom_{\fh}( \C_\lambda,H^p(\fn_0,H^q(\fg_1,L(\mu)))=0 \quad\text{for}\quad p>0.$$
Hence the spectral sequence collapses and we obtain
$$\Hom_{\fh}( \C_\lambda,H^j(\fn, L(\nu)))=\Hom_{\fb_0}(\C_\lambda,H^j(\fg_1,L(\nu)))=\Hom_{\fg_0}(L_0(\lambda),H^j(\fg_1,L(\nu))),$$
since $H^j(\fg_1,L(\nu))$ is a semisimple $\fg_0$-module.
By the analogue of equation~\eqref{VerH} in the category~$\cF$ we have
$$\dim\Ext_{\cO}^j(K(\lambda),L(\nu))=\dim\Hom_{\fg_0}(\L_0(\lambda),H^j(\fg_1,L(\nu))).$$
Hence the statement follows.
\end{proof}

\subsection{Further vanishing properties of extensions}
In this subsection we continue to consider $\fg=\mathfrak{gl}(m|n)$. We assume that the Borel subalgebra is the distinguished or anti-distinguished one. The essential (and characterising) property of these choices is that every positive root which is simple in $\Delta^+_{\oa}$ is also simple in $\Delta^+$.

\begin{lemma} 
\label{lemtwist}
Consider $\lambda,\mu\in\fh^\ast$ and a simple reflection $s\in W$.
\begin{enumerate}[$($i$)$]
\item If $s\cdot\lambda=\lambda$ and $s\cdot\mu <\mu$, we have $\Ext_{\cO}^{\bullet}(M(\lambda),L(\mu))=0$.
\item If $s\cdot\lambda<\lambda$ and $s\cdot\mu<\mu$, we have $\Ext^j_{\cO}(M(s\cdot\lambda),L(\mu))\simeq \Ext^{j-1}_{\cO}(M(\lambda),L(\mu))$.
\end{enumerate}
\end{lemma}
\begin{proof}
We prove these results using the right exact twisting functors $T_s$ and the left derived functors $\cL_iT_s$ on category~$\cO$, as studied in~\cite{CMW, CouMaz}. For both $(i)$ and $(ii)$ we have $$\cL_iT_s\left(M(\lambda)\right)\simeq\delta_{i,0}M(s\cdot\lambda)\qquad\mbox{and}\qquad \cL_iT_s( L(\mu))\simeq\delta_{i,1}L(\mu),$$ see Lemmata 5.4 and 5.7 and Theorem 5.12$(i)$ in~\cite{CouMaz}. The combination of the displayed equations with Proposition 5.11 in~\cite{CouMaz}, leads to
$$\Ext^j_{\cO}(M(s\cdot\lambda),L(\mu))\simeq \Ext^{j-1}_{\cO}(M(\lambda),L(\mu)).$$
A step-by-step explanation of this procedure can be found in the proof of Proposition 3 in~\cite{MR2366357}. This yields $(ii)$. In case $s\cdot\lambda=\lambda$, we obtain by iteration that $$\Ext^j_{\cO}(M(\lambda),L(\mu))=\Hom_{\cO}(M(\lambda),L(\mu))=0,$$ proving $(i)$.
\end{proof}
An alternative proof of Lemma~\ref{lemtwist} follows from equation~\eqref{VerH} and a HS spectral sequence reducing the statement to an $\mathfrak{sl}(2)$ property.

\begin{remark}
The combination of Lemmata \ref{lemtwist} and \ref{lemOF} completely determines the KL polynomials $p_{\lambda,\nu}$ in case 
$\nu\in\Lambda^{++}$ in terms of the KL polynomials for $\cF$. 
\end{remark}

Next we prove that some KL polynomials can be obtained from the ones for the underlying Lie algebra.
\begin{lemma}\label{KLorbit}
For $\mu,\lambda\in \fh^\ast$ in the same orbit of the Weyl group we have
\begin{eqnarray*}\dim\Ext^j_{\cO}(M(\mu),L(\lambda))&=&\dim\Ext^j_{\cO^0}(M_0(\mu),L_0(\lambda))\qquad\mbox{and}\\
\dim\Ext^1_{\cO}(L(\mu),L(\lambda))&=&\dim\Ext^1_{\cO^0}(L_0(\mu),L_0(\lambda)).
\end{eqnarray*}
\end{lemma}
\begin{proof}
The second equality follows from the first and equation~\eqref{extsimple}.

Now we prove the first equality. We use the element $z\in \fh$, defined in equation~\eqref{gradel}. Note that $\fh$ and therefore $z$ act on 
$H^{p+q}(\fn,L(\lambda)$ as well as on the spectral sequence term $H^p(\fn_0,H^q(\fg_1,L(\lambda)))$. Recall that $H^q(\fg_1,L(\lambda)))$ is a
$\fg_0$-subquotient of $S^q\fg_{-1}\otimes\Lambda\fg_{-1}\otimes L_0(\lambda)$. Therefore, if $q>0$, then
any weight $\nu$ of $H^q(\fg_1,L(\lambda)))$ satisfies
the condition $\nu(z)<\lambda(z)$.
Next we notice that $\lambda(z)=\mu(z)$. Therefore,
$$\Hom_{\fh}(\C_\mu,H^p(\fn_0,H^q(\fg_1,L(\lambda))))=0\qquad\mbox{if}\quad q>0.$$
This yields
$$\Hom_{\fh}(\C_\mu,H^j(\fn,L(\lambda)))\simeq\Hom_{\fh}(\C_\mu,H^j(\fn_0,L(\lambda)^{\fg_1}))\simeq \Ext_{\cO^0}^j(M_0(\mu),L_0(\lambda)), $$
since $L(\lambda)^{\fg_1}\simeq L_0(\lambda)$, concluding the proof. An alternative proof of the first equality follows from the algorithm to calculate the canonical basis in Section 3 of~\cite{Brundan3}.
\end{proof}

The following vanishing lemma will be useful in Section \ref{secfpd}.
\begin{lemma}
\label{vancentre}
Consider $\lambda,\mu\in\fh^\ast$, then we have
$$\Ext^j_{\cO}(M(\lambda),L(\mu))\not=0\quad\Rightarrow \quad \max\{j-l(w_0),0\}\, \le \,\mu(z)-\lambda(z)\, \le\, j+\dim\fg_1,$$
with $z\in\mathfrak{z}(\fg_0)\subset\fh$ as in equation~\eqref{gradel} and $w_0$ standing for the longest element in $W$.
\end{lemma}
\begin{proof}
We use the reformulation~\eqref{VerH} and the spectral sequence \eqref{HSss}. This implies that the extension vanishes unless $\lambda=w\cdot\nu$,
for some $w\in W$ and a highest weight $\nu$ of a simple $\fg_0$-subquotient of $H^q(\fg_1,L(\mu))$. Moreover,
$j-q\le l(w_0)=\dim \fn_0$.
In particular we have
$$[S^q(\fg_{-1})\otimes \Lambda(\fg_{-1})\otimes L_0(\mu):L_0(w^{-1}\cdot\lambda)]\not=0.$$
Hence $q\leq \mu(z)-\lambda(z)$. This implies 
$$j-l(w_0)\, \le\, q\,\le \,\mu(z)-\lambda(z)$$
and 
$$j+\dim\fg_1\geq q+\dim\fg_1\geq  \mu(z)-\lambda(z).$$
The proposed inequalities thus follow.
\end{proof}

\subsection{Socle of the tensor space}
\begin{theorem}
The socle of the $\mathfrak{sl}(\infty)$-module $\mathbb V^{\otimes m}\otimes \mathbb W^{\otimes n}$ contains $\{b_{\mu}\,|\,\mu\in \Lambda\}$.

Furthermore, under the $\mathfrak{sl}(\infty)$-module morphism $\mathbb V^{\otimes m}\otimes \mathbb W^{\otimes n}\simeq K(\cO^{\Delta}_{\mZ})$, this socle corresponds to the subgroup of $K(\cO^{\Delta}_{\mZ})$ generated by the projective modules.
\end{theorem}
\begin{proof}
We denote the socle by $\mathbb V^{\{m,n\}}$. It is proved in Theorem 2.2 of~\cite{Penkov} that $\mathbb V^{\{m,n\}}$ corresponds to the 
intersection of the kernels of all contraction maps $$ \mathbb V^{\otimes m}\otimes \mathbb W^{\otimes n}\tto \mathbb V^{\otimes m-1}\otimes \mathbb W^{\otimes n-1}.$$ 
So ${v}_\lambda$ is in $\mathbb V^{\{m,n\}}$ if and only if $\lambda$ is typical. 

According to the finiteness discussion of the algorithm to compute Lusztig's canonical basis in Section 3 of~\cite{Brundan3} it follows that for any $\nu\in\Lambda$,
$$\dot{b}_\nu=\sum_{\lambda} f_\lambda(q) A_\lambda \dot{v}_\lambda,$$
for a finite sum of typical (dominant) $\lambda\in\Lambda$, certain $A_\lambda\in U_q(\mathfrak{sl}(\infty))$ and $f_{\lambda}(q)\in\mZ[q,q^{-1}]$. Evaluating this in $q=1$ yields the first part. 

To prove the second statement we consider the description of the socle in Theorem~2.1 of~\cite{Penkov}. This implies that the socle is the direct sum of highest weight modules (with respect to the system of positive roots introduced there), where the highest weight vectors are linear combinations of $v_\lambda$ for typical $\lambda\in\Lambda$. This implies that the socle is inside the submodule generated by the basis $\{b_{\mu}\,|\,\mu\in \Lambda\}$.
\end{proof}

\section{Finiteness of homological dimension and the associated variety}
\label{secfpd}
In this section (except Lemma~\ref{propind}) $\fg$ is in the list \eqref{listA}. In Section 4.4 of~\cite{preprint}, Mazorchuk proved that the finitistic global dimension of (parabolic) category~$\cO$ is finite for classical Lie superalgebras. In this section, we relate the finiteness of projective dimensions with the associated variety of \cite{Duflo, Vera}. For any $M\in\fg$-mod we define its associated variety $X_M$ as
\begin{equation}\label{defasva}
\begin{aligned}
&X=\{x\in\fg_{\ob}\,\,\mbox{with}\,\, [x,x]=0 \},\\
&X_M=\{x\in X\,\,\mbox{with}\,\,  xM\not=\ker_xM\}.
\end{aligned}\end{equation}

Let  $\cA_0$ denote the additive and Karoubian closure of the category of modules of the form $\Ind N_0$ with $N\in \cO^0$. 
This category does not need to be abelian. By iteration we define $\cA_j$ as the full subcategory of $\cO$, containing the modules in $\cA_{j-1}$ as 
well as modules in $\cO$ which can be written as a kernel or cokernel of an injective or surjective morphism between two modules in $\cA_{j-1}$ or an 
extension of two modules in $\cA_{j-1}$. By taking the full subcategory of $\cO$ consisting of modules in some $\cA_j$ for $j\in\N$, we obtain an abelian 
full Serre subcategory which we denote by~${}^{(\fg,\fg_{\oa})}\cO$. The equivalence between $(iii)$ and $(iv)$ in the following theorem 
actually implies firstly that ${}^{(\fg,\fg_{\oa})}\cO\simeq \cA_{2l(w_0)}$ and secondly that this category can also be obtained by a similar procedure using only cokernels.

\begin{theorem}
\label{finpdind}
Let $\fg$ be a Lie superalgebra from the list ~\eqref{listA} and $M\in\cO$. Denote by $\pd_{\cO}$ the projective dimension in $\cO$. The following conditions are 
equivalent:
\begin{eqnarray*}
(i)\quad \pd_{\cO}M<\infty;\quad&(iii)&  M\in {}^{(\fg,\fg_{\oa})}\cO;\\
(ii)\;\,\quad X_M=\{0\};\quad&(iv)& \pd_{\cO}M\le 2l(w_0). 
\end{eqnarray*}
Consequently we have $\fd\cO=2l(w_0)$.
\end{theorem}
The statement on the finitistic global dimension will be improved upon, by determining it each for each indecomposable block individually in Theorem \ref{fdblock}.
\begin{remark}
As the proof of Theorem \ref{finpdind} reveals, for arbitrary basic classical Lie superalgebra we still have the property $$(i)\Leftrightarrow (iii)\Leftrightarrow (iv) \Rightarrow (ii).$$
However, already for $\mathfrak{sl}(1|1)$ we have $(ii)\not\Rightarrow (i)$, by Remark \ref{error}.
\end{remark}

The remainder of this section is devoted to proving this theorem.

\begin{lemma}
\label{propind}
Let $\fg$ be a classical Lie superalgebra, $M\in\cO$ and $N_0\in \cO^0$. We have
\begin{enumerate}[$($i$)$]
\item $\pd_{\cO}M\ge \pd_{\cO^0}\Res M$;
\item If $M$ is a direct summand of $\Ind N_0$, then $\pd_{\cO}M\le \pd_{\cO^0}N_0$;
\item $\pd_{\cO}\Ind N_0=\pd_{\cO^0}N_0$.
\end{enumerate}
Consequently we have $\fd\cO=2l(w_0)$.
\end{lemma}
\begin{proof}
The functors $\Res$ and $\Ind$ are exact and map projective modules to projective modules, this implies $(i)$ and $(ii)$.

Claim $(iii)$ follows from combining $(i)$ and $(ii)$ since the $\fg_{\oa}$-module $\mC$ is a direct summand of $\Lambda\fg_{\ob}$, so $N_0$ is a direct summand $\Res \Ind N_0\simeq \Lambda\fg_{\ob}\otimes N_0$.

Property $(iii)$ implies $\fd\cO\ge 2l(w_0)$ and the reversed inequality is Theorem~3 of~\cite{preprint}.
\end{proof}

\begin{proposition}
\label{propflag}
Take $\fg$ to be in~\eqref{listA}. Assume that $M\in\cO$, is $\fg_{-1}$-free (resp. $\fg_{1}$-free), then $M$ has a Kac flag (resp. dual Kac flag).
\end{proposition}
\begin{proof}
Take $M$ to be $\fg_{-1}$-free. The $\fg_0$-module $N:=M/\fg_{-1}M$ decomposes according to the eigenvalues of $z$, in equation~\eqref{gradel}, as 
$N=\bigoplus_{\alpha\in\R}N_\alpha.$ Since $M$ is finitely generated there is only a finite amount of $\alpha$ for which $N_\alpha\not=0$. We take $\alpha_0$ to be the highest of these. Then $N_{\alpha_{0}}$ is isomorphic to a $\fg_0\oplus\fg_1$-submodule of $\res^{\fg}_{\fg_0\oplus\fg_1}M$. Since $M$ is $\fg_{-1}$-free, we find 
$$U(\fg)\otimes_{U(\fg_0\oplus \fg_{1})}N_{\alpha_0}\hookrightarrow M.$$
Since $U(\fg)\otimes_{U(\fg_0\oplus \fg_{1})}N_{\alpha_0}$ clearly has a filtration by Kac modules, the proof can be completed iteratively by considering the cokernel of 
the above morphism.

The proof for a $\fg_1$-free $M$ is identical.
\end{proof}

\begin{remark}\label{error}
For $\fg=\mathfrak{sl}(n|n)$, the element $z\in\mathfrak{z}(\fg_0)$ does not exist, which leads to counterexamples of Proposition \ref{propflag} and Theorem \ref{finpdind}. Consider $\mathfrak{sl}(1|1)=\langle x,y,e\rangle$ with 
$$[x,x]=0=[y,y],\,\quad [x,y]=e,\,\quad [e,x]=0=[e,y].$$ The two dimensional module $V=\langle v_1,v_2\rangle$ with action $xv_1=v_2=yv_1$ 
and with trivial action of $e$ satisfies $X_V=\{0\}$ but has no (dual) Kac flag and is not projective in $\cF=\cO$.
\end{remark}

\begin{lemma}
\label{finpdflag}
If $M\in\cO$ admits a Kac flag and a dual Kac flag, then $\pd_{\cO}M< \infty$.
\end{lemma}
\begin{proof}
We will prove $\id_{\cO}M<\infty$, for the injective dimension, which is equivalent by Section 3 in~\cite{preprint}. 
The proof could also be done immediately for projective dimension using an unconventional definition of the Kac modules. For any $\lambda,\mu\in\fh^\ast$, we prove
$$\Ext^j_{\cO}(K(\mu),\overline{K}(\lambda))=0=\Ext^j_{\cO}(\overline{K}(\lambda), K(\mu))\qquad\mbox{for}\quad j>2l(w_0).$$
Indeed, applying Frobenius reciprocity twice (using Theorem~25$(i)$ in~\cite{CouMaz2} and equation~\eqref{indcoind}) yields
$$\Ext^j_{\cO}(\overline{K}(\lambda), K(\mu))=\Ext^j_{\cO^0}(L_0(\lambda),L_0(\mu-2\rho_1)),$$
and a similar argument holds for the other equality.

 In particular this implies that for any $\lambda,\mu\in\fh^\ast$ we have
\begin{equation}\label{vanishM}\Ext^j_{\cO}(K(\mu),M)=0=\Ext^j_{\cO}(\overline{K}(\lambda), M)\qquad\mbox{for}\quad j>2l(w_0).\end{equation}
Since $M$ is finitely generated, the element $z\in\mathfrak{z}(\fg_0)$ in equation~\eqref{gradel} has eigenvalues in an interval of finite length $p$. We prove that for any $\alpha\in\fh^\ast$,
$$\Ext^j_{\cO}(L(\alpha),M)=0\qquad \mbox{if}\qquad j>p+\dim\fg_1+4l(w_0).$$
Assume that the above extension would not be zero. The short exact sequence $N\hookrightarrow K(\alpha)\tto L(\alpha)$ and the vanishing properties in~\eqref{vanishM} imply that there must be a $\beta\in\fh^\ast$ for which $L(\beta)$ is a subquotient of $N$ (and therefore satisfies $\beta(z)\le \alpha(z)-1$) such that
$$\Ext^{j-1}_{\cO}(L(\beta),M)\not=0.$$
This procedure and the dual one using dual Kac modules can be repeated so that we come to the conclusion that there must exist $\kappa,\nu\in \fh^\ast$ with $\kappa(z)\ge\alpha(z)+j-2l(w_0)$ and $\nu(z)\le \alpha(z)-j+2l(w_0)$ such that
\begin{equation}\label{ext2l}\Ext^{2l(w_0)}_{\cO}(L(\kappa),M)\not=0\qquad\mbox{and}\qquad \Ext^{2l(w_0)}_{\cO}(L(\nu),M)\not=0.\end{equation}
The combination of equation~\eqref{extsimple} and Lemma~\ref{vancentre} yields that for any two $\mu,\mu'$ we have
\begin{equation*}\Ext^i_{\cO}(L(\mu),L(\mu'))=0\qquad \mbox{unless}\qquad |\mu(z)-\mu'(z)|\le \dim\fg_1+i.\end{equation*}

Equation~\eqref{ext2l} therefore implies that both $\kappa(z)$ and $\nu(z)$ must lie in an interval of length $p+2(\dim\fg_1+2l(w_0))$. However, the construction above implies that
$$\kappa(z)-\nu(z)\ge 2j-4l(w_0),\quad \mbox{with} \quad j>p+\dim\fg_1+4l(w_0).$$
This means we have proved that 
$$\id_{\cO}M < p+\dim\fg_1+4l(w_0),$$
for some finite $p\in\N$.
\end{proof}

\begin{lemma}\label{auxass1} Let $\mathfrak l=\fl_{\ob}$ be a finite dimensional abelian Lie superalgebra with trivial even part, equipped with a $\mZ$-grading 
$\mathfrak l=\mathfrak l_0\oplus\mathfrak l_1$. Let $M$ be a graded $\mathfrak l$-module (may be infinite dimensional) and assume there exists $k\in\mZ$ such that $M_j=0$ for 
$j\geq k$. If $M$ is free over $\mathfrak l_0$ and $\mathfrak l_1$, then it is also free over $\mathfrak l$.
\end{lemma}
\begin{proof} Let $\{v_1,\dots,v_p\}$ be a basis of $\mathfrak l_0$ and $\{u_1,\dots,u_q\}$ a basis of $\mathfrak l_1$, set
$v=v_1\dots v_p$ and  $u=u_1\dots u_q$, both are elements in $U(\mathfrak l)=\Lambda(\mathfrak l)$. 

Let $M'\subset M$ be a maximal graded subspace such that $M'\xrightarrow{vu}vu M$ is an isomorphism and $N$ be the submodule generated by $M'$. Then $N$ is 
free over $\mathfrak l$, therefore $N$ is both projective and injective. Let $M''=M/N$, then $M\simeq M''\oplus N$. We claim that $M''=0$.
Assume the opposite. Note that $M''$ satisfies all the assumptions of the lemma and $vuM''=0$. 
Pick up the maximal $j$ such that
$M''_j\neq 0$. Then $M''_j$ is a free $\mathfrak l_0$-module and one can find $m\in M''_j$ such that $vm\neq 0$. On the other hand, $\mathfrak l_1 m=0$. Since
$M''$ is free over $\mathfrak l_1$ we obtain that $m=u m'$ for some $m'\in M''$. This implies that $v u m'\neq 0$, which is a contradiction.
\end{proof}

\begin{lemma}\label{auxass2} Consider $M\in \cO$ which satisfies $X\cap \fg_{\pm1}=\{0\}$. Then $M$ is $\fg_{\pm1}$ free.
\end{lemma}
\begin{proof}
We will prove the statement if $X\cap \fg_{1}=\{0\}$. The case $X\cap \fg_{-1}=\{0\}$ is similar. Consider $\fg=\mathfrak{gl}(m|n)$ with $m>1$. 
The proof is by induction on $m$, using the previous 
lemma. 
Define the grading on $\fg_1$ and on $M$ by setting the degree of the root space $\fg_{\alpha}$ to be $(\alpha,\varepsilon_1)$ and the degree of the weight space
$M_\lambda$ to be $(\lambda,\varepsilon_1)$. Then $\mathfrak l:=\fg_1$ and $M$ satisfy all the conditions of Lemma \ref{auxass1}. 

Now $\fl_0\subset\mathfrak{gl}(m-1|n)$, with $\fl_0=\mathfrak{gl}(m-1|n)_1$. As a $\mathfrak{gl}(m-1|n)$-module, every graded component of $M$, with respect to 
the above grading, is a direct summand of $M$ and an object in category $\cO$ for $\mathfrak{gl}(m-1|n)$. By the induction hypothesis, $M$ is thus free over $\fl_0$.

Similarly, $\fl_1=\mathfrak{gl}(1|n)_1$. Again we find that $M$, now regarded as a $\mathfrak{gl}(1|n)$-module, 
decomposes into modules belonging to category $\cO$. 
For this one could consider the graded components with respect to the grading given by 
$\lambda\mapsto -\sum_{i=2}^m(\lambda,i\varepsilon_i)$. 
Again by the induction assumption $M$ is free over $\fg_1$. 

Lemma \ref{auxass1} then implies that $M$ is free over $\fg_1$. The base case $m=1$ can be covered by similar induction on $n$.
\end{proof}

\begin{proof}[Proof of Theorem \ref{finpdind}]
The equivalence of $(i)\Leftrightarrow (iv)$ follows from Lemma~\ref{propind}.

Assume  that $M\in\cO$ is a direct summand of a module induced from one in~$\cO^0$. Lemma~\ref{propind}$(ii)$ therefore yields $\pd_{\cO}M<\infty$. Recall the categories $\cA_j$ from the beginning of this section. Lemma~6.9 in~\cite{MR2428237} or Section 2.3 in~\cite{CouMaz2} shows that if every module in $\cA_{j-1}$ has finite projective dimension in $\cO$, then so has every module in $\cA_j$. Hence we find $(iii)\Rightarrow (i)$.

Now assume that $\pd_{\cO}M<\infty$ holds. Since $M$ has a finite resolution by projective modules and projective modules are direct summands of modules induced from projective modules in $\cO^0$ we obtain $M\in{}^{(\fg,\fg_0)}\cO$. This proves $(i)\Rightarrow (iii)$. 

By Lemma \ref{auxass2}, $X_M\cap\fg_{\pm 1}=\{0\}$ implies that $M$ is $\fg_{ \pm 1}$-free, property $(ii)\Rightarrow (i)$ follows from Proposition \ref{propflag} and Lemma~\ref{finpdflag}. Since every projective module is a direct summand in a module induced from $\fg_0$, it has trivial associated variety, see Lemma 2.2(1) in~\cite{Duflo}. If $M$ has a finite resolution by projective modules, it has finite projective dimension as a $\mC [x]$-module for any $x\in\fg_{\ob}$ with $[x,x]=0$, so it is $\C [x]$-free. Thus $X_M=\{0\}$ and we obtain $(i)\Rightarrow (ii)$.

The last statement is a special case of Lemma~\ref{propind}.
\end{proof}


\section{$B_{\oa}$-orbits in the self-commuting cone and some results on associated variety in category~$\cO$}
\label{secassvar}

The associated variety $X_M$ in \eqref{defasva} of an object $M$ in $\cO$ is intrinsically not a categorical invariant, contrary to projective dimension and complexity, although results as Theorem~\ref{finpdind} indicate interesting links with categorical invariants. However, it is thus not possible to use the equivalences of categories in~\cite{CMW} to reduce to the integral case. Therefore, throughout the entire section, we consider weights in $\fh^\ast$, not just in $P_0$.

In the first two subsections we will consider the associated variety of modules in category~$\cO$ for $\fg$ a basic classical Lie superalgebra in the list
\begin{equation}\label{list}
\mathfrak{sl}(m|n), m\neq n;\,\,\mathfrak{gl}(m|n);\,\,  \mathfrak{osp}(m|2n);\,\, D(2,1,\alpha);\,\, G(3);\,\, F(4),\end{equation}
with arbitrary Borel subalgebra. In the last three subsections we will focus on $\fg=\mathfrak{gl}(m|n)$ with the distinguished Borel subalgebra.

\subsection{$B_{\oa}$-orbits} Let $B_{\oa}$ be the Borel subgroup of the algebraic group $G_{\oa}$ 
with Lie algebra~$\fb_{\oa}$. The group $B_{\oa}$ acts on $X$ of \eqref{defasva} by adjoint action. If $M$ is in category~$\cO$, then the simply connected cover of 
$B_{\oa}$ acts on $M$. Thus the associated variety $X_M$, as defined in equation~\eqref{defasva}, is a $B_{\oa}$-invariant subvariety of $X$. Therefore it is 
important to study $B_{\oa}$-orbits in~$X$. It is proven in~\cite{Duflo} that $X$ has finitely many $G_{\oa}$-orbits. We will show in this subsection that the same is true for $B_{\oa}$-orbits.

Let $S=\{\alpha_1,\dots,\alpha_k\}$ be a set of mutually orthogonal linearly independent isotropic roots and $x_1,\dots, x_k$ be some non-zero elements in the
root subspaces $\fg_{\alpha_1},\dots,\fg_{\alpha_k}$ respectively. Then $x_S:=x_1+\dots+x_k\in X$. For such an $x_S\in X$, we say that its rank is $k=|S|$. Let $\cS$ denote the set of all subsets $S$ of mutually orthogonal
linearly independent isotropic roots and $X/B_{\oa}$ denote the set of $B_{\oa}$-orbits in~$X$. In Theorem 4.2 of \cite{Duflo} it was proved that $X/G_{\oa}\,\simeq\, \cS/{W}$, now we derive an analogous description of $X/B_{\oa}$.

Define the map 
$$\Phi:\cS\to X/B_{\oa};\,\quad \Phi(S):=B_{\oa} x_S\\,\, \mbox{ for all }S\in\cS.$$
We assume that $\Phi(\emptyset)=0$. Note that $\Phi$ does not depend on a choice of $x_1,\dots,x_k$, since any two such elements are conjugate under the action of a maximal torus in $G_{\oa}$.

\begin{theorem}\label{mainass}
Consider $\fg$ in the list \eqref{list}, the map $\Phi$ is a bijection, so $ X/B_{\oa}\,\simeq\,\cS$.
\end{theorem}  
\begin{proof} First, we will prove that $\Phi$ is surjective, {\it i.e.} that every~$B_{\oa}$-orbit contains an $x_S$ for some $S\in \cS$. 
Recall that Theorem 4.2 in~\cite{Duflo} implies that every~$G_{\oa}$-orbit contains an $x_S$ for some $S\in\cS$. Due to the Bruhat decomposition
$$G_{\oa}=\bigsqcup_{w\in W}B_{\oa}wB_{\oa},$$
it suffices to prove that for every~$S\in\cS$ and $w\in W$, $B_{\oa}wB_{\oa}x_S$ is a union of $B_{\oa}x_{S'}$ for some $S'\in\cS$.
Moreover, using induction on the length of $w$ (and $B_{\oa}swB_{\oa}x_S\subset B_{\oa}s B_{\oa} w B_{\oa}x_S$ for a simple reflection $s$), it is sufficient to prove the latter statement only in the case when~$w=r_{\alpha}$ is a simple reflection.

Let $G_\alpha$ be the $SL_2$-subgroup in $G_{\oa}$ associated with the root $\alpha$ and $B_\alpha:=G_\alpha\cap B_{\oa}$. Since 
$$B_{\oa}r_{\alpha}B_{\oa}\subset B_{\oa}G_\alpha,$$
we have to show that $G_\alpha x_S$ lies in a union of $B_{\oa}x_{S'}$ for some $S'\in\cS$.

We need the following well-known facts about root system of the superalgebras in \eqref{list}. By odd $\alpha$-chain we mean the maximal subset of odd roots of the form $\beta+p\alpha$ with $p\in\Z$ for some $\beta\in\Delta$.
\begin{enumerate}
\item The length of any odd $\alpha$-chain consisting of odd roots is at most $3$;
\item If $\beta$ is an isotropic root then either $\beta+\alpha$ is not a root or $\beta-\alpha$ is not a root;
\item If $S\in \cS$, then at most $2$ roots of $S$  are not preserved by $r_{\alpha}$.
\end{enumerate}

We consider the three cases allowed by statement $(3)$ individually. If all roots of $S$ are preserved by $r_\alpha$, then $G_\alpha x_S=x_S$ and the statement is trivial.

Assume that there exists exactly one root $\alpha_i\in S$, which is not preserved by $r_{\alpha}$. Consider the $G_{\alpha}$-submodule $V_i\subset \fg_{\ob}$ generated 
by $x_i$.
By (2) $x_i$ and $r_{\alpha}(x_i)$ are the lowest and the highest weight vectors in $V_i$ and from representation theory of $SL_2$ we have
$$G_{\alpha}x_i=B_{\alpha}x_i\cup B_{\alpha}r_{\alpha}(x_i).$$ Since for any $\alpha_j\in S$ with $j\neq i$ we have $G_{\alpha}x_j=x_j$, we obtain
$$G_{\alpha}x_S=B_{\alpha}x_S\cup B_{\alpha}x_{r_{\alpha}(S)}\subset B_{\oa} x_S\cup B_{\oa} x_{r_\alpha(S)}.$$ Hence the statement is proved in this case.

Finally, assume that there are two distinct roots $\alpha_i,\alpha_j\in S$ which are not preserved by $r_{\alpha}$. Since this case is only possible for 
Lie superalgebras  of defect greater than $1$, we may assume that $\fg$ is either general linear or orthosymplectic. In this case 
$\alpha_i=\pm\varepsilon_a\pm\delta_b$ and $\alpha_j=\pm\varepsilon_c\pm\delta_d$ for some $a\neq c$ and $b\neq d$. That implies that $\alpha_i,\alpha_j$ and
$\alpha$ are roots of some root subalgebra $\fg'$ isomorphic to $\mathfrak{sl}(2|2)$. Furthermore, the corresponding subgroup $G'_{\oa}$ preserves $x_l$ for
all $l\neq i,j$. Therefore it suffices to check the analogous statement for $G'$. It can be done by direct 
computation and we leave it to the reader.  

Now we will prove that $\Phi$ is injective, {\it i.e.} $x_{S'}\in B_{\oa}x_S$ implies $S=S'$. Assume that $x_{S'}=\operatorname{Ad}_g x_S$ for some $g\in B_{\oa}$. 
Note $|S|=|S'|$ by the property $X/G_{\oa}\simeq\cS/{W} $ of~\cite{Duflo}. Let
$$\mathfrak m=\bigoplus_{\alpha\in S}\fg_\alpha,\quad \mathfrak m'=\bigoplus_{\alpha\in S'}\fg_\alpha,\quad \fh'=\operatorname{Ad}_g(\fh).$$
Then $\mathfrak m$ (resp. $\mathfrak m'$) is the minimal $\fh$-submodule of $\fg$ containing $x_S$ (resp. $x_{S'}$). On the other hand, $\mathfrak m'$ is also the minimal
$\fh'$-submodule containing $x_{S'}$. Furthermore $\operatorname{Ad}_gh-h\in \fn^+_{\oa}$ for all $h\in \fh$, which implies that the spectra of $h$ and of
$\operatorname{Ad}_gh$ in $\mathfrak m'$ coincide. Since the former is $\{\alpha(h)\,|\,\alpha\in S\}$ and the latter is $\{\alpha(h)\,|\,\alpha\in S'\}$,
we obtain $S=S'$.
\end{proof}

For any $\fg$-module $M\in\cO$, we introduce the notation
$$\cS(M)=\{S\in\cS | \Phi(S)\subset X_M\}.$$
In particular, Theorem \ref{mainass} implies that for $M,N\in\cO$ we have $\cS(M)=\cS(N)$ if and only if $X_M=X_N$.

\subsection {General properties of the associated variety for modules in category~$\cO$} Recall from Lemma~6.2 of \cite{Duflo} 
that if $x\in X$ and $M$ is a $\fg$-module, then 
$M_x:=\ker x/\im\, x$ is $\fg_x$-module, where~$\fg_x:=\ker\ad x/\im\,\ad x$. If $\fg$ is a basic classical superalgebra, then  
$\fg_x$ is also a basic classical superalgebra.
For example, if $\fg=\mathfrak{gl}(m|n)$ and the rank of $x$ is $k$,
then $\fg_x$ is isomorphic to $\mathfrak{gl}(m-k|n-k)$. If $x=x_S$ we identify~$\fg_x$ with the root subalgebra in $\fg$, whose roots are orthogonal 
but not proportional to the roots from $S$.

For a Lie superalgebra $\fl$ we denote by $Z(\fl)$ the center of the universal enveloping algebra~$U(\fl)$ and by
$\check{Z}(\fl)$ the set of central characters. In Section 6 of \cite{Duflo}, a map $\eta:Z(\fg)\to Z(\fg_x)$ was introduced. In Theorem 6.11 of  {\it op. cit.}, it was proved that all fibres of the
dual map $\check\eta:\check{Z}(\fg_x)\to \check{Z}(\fg)$ are finite, more precisely any fibre consists of at most two points. If $M$ admits a 
generalised central character $\chi$ then $M_x$ is a direct sum of $\fg_x$-modules admitting generalised central characters from $\check\eta^{-1}(\chi)$. The degree of atypicality of the central characters in $\check\eta^{-1}(\chi)$ is equal to the degree of atypicality of $\chi$ minus the rank of $x$, see e.g. equation (3) in~\cite{Vera}. In the case $\fg=\mathfrak{gl}(m|n)$, the map $\check\eta$ is injective and it maps a central character $\chi'$ of $\fg_x$ to the central character $\chi$ of
$\fg$ with the same core. 

Below we summarise the general properties of the functor $\fg$-mod to $\fg_x$-mod, which sends $M$ to $M_x$.
\begin{lemma}\label{general}Consider $\fg$ in the list \eqref{list}.
\begin{enumerate}
\item For any $\fg$-modules $M$ and $N$ we have $(M\otimes N)_x\simeq M_x\otimes N_x$.
\item Let $M$ be a $\fg$-module from the category~$\cO$ which admits a generalised central character with atypicality degree $k$ and $S\in\cS$. If $S\in \cS(M)$, then
$|S|\leq k$.
\item Let $M$ be from the category~$\cO$, $x=x_S$ for some $S\in\cS$ and $\fb_x:=\fb\cap\fg_x$, then $\fb_x$ acts locally finitely on $M_x$ and $M_x$ is a weight module.  
\end{enumerate}
\end{lemma} 
\begin{proof} To prove (1) note that we have the obvious homomorphism $M_x\otimes N_x\to (M\otimes N)_x$ of $\fg_x$-modules. To check that it is an isomorphism
consider $M$ and $N$ as $\mathbb C[x]$-modules. Then
$$M\simeq M^f\oplus M_x,\quad N\simeq N^f\oplus N_x,$$ 
where~$M^f$ and $N^f$ are free $\mathbb C[x]$-modules. Since $M^f\otimes N$ and $M\otimes N^f$ are free we obtain an isomorphism
 $M_x\otimes N_x\simeq (M\otimes N)_x$.

To show (2) use the map $\check\eta$. If $|S|>k$, then $\check\eta^{-1}(\chi)$ is empty and therefore $M_x=0$. 

Finally, (3) is trivial since $M_x$ is a subquotient of $M$, and $\fb_x$ acts locally finitely and $\fh\cap \fg_x$ diagonally on $M$.
\end{proof}
\begin{remark} We believe that (3) can be strengthened. Namely, if $M$ lies in the category~$\cO$, then $M_x$ belongs to the category~$\cO$ for $\fg_x$, {\it i.e.} $M_x$
is finitely generated. But we do not have a proof of this at the moment.
\end{remark}

\subsection{On associated variety of Verma modules for $\mathfrak{gl}(m|n)$} 
In this subsection we assume $\fg=\mathfrak{gl}(m|n)$ and consider the distinguished Borel subalgebra
$\fb=\fb_{\oa}\oplus \fg_1$. We set $\cS(\lambda):=\cS(M(\lambda))$.
We start with the following technical lemma.

\begin{lemma}\label{technical} Let $\fs$ be $(1|2)$ dimensional superalgebra with odd generators $\xi,\eta$ and even generator $u$ satisfying
$[\xi,\eta]=0$, $[u,\xi]=\xi$ and $[u,\eta]=-\eta$. Assume that $M$ is an $\fs$-module semisimple over $\mC u$ and such that the spectrum of $u$ in $M$ is bounded
from above, {\it i.e.} there exists $\gamma_0$ such that $\operatorname{Re}\gamma<\gamma_0$ for any eigenvalue $\gamma$ of $u$. Then
$M_{\eta}=0$ implies $M_{\eta+\xi}=0$.
\end{lemma}
\begin{proof} Note that any $u$-eigenvector $v\in M$ such that $\xi\eta v\neq 0$ generates  a projective $\fs$-submodule 
(in the category of $\fs$-modules semisimple over $u$). Hence we have a decomposition $M=P\oplus L$ for some projective $P$ and $L$ such that
$\xi\eta L=0$. Obviously, $P_{\xi+\eta}=0$ and we have to check only that $L_{\xi+\eta}=0$.  Since $L_\eta=0$, it follows that $L$ is free over $\eta$, hence
we have a $u$-invariant decomposition $L=L'\oplus L''$ such that $\xi L'\subset L''$, $\eta L'=L''$, 
$\xi L''=\eta L''=0$ and $\eta:L'\to L''$ is an isomorphism. Let $\zeta: L''\to L'$ be the inverse of $\eta$. Then there are $n'\in \End(L')$ and $n''\in\End(L'')$ such that
$$\zeta(\xi+\eta)=\operatorname{Id}_{L'}+n',\quad (\xi+\eta)\zeta=\operatorname{Id}_{L"}+n'',$$ 
where~$[u,n']=2n'$ and $[u,n'']=2n''$. The latter condition and the assumption on the spectrum of $u$ imply that
$n'$ and $n''$ are locally nilpotent and hence $\operatorname{Id}_{L'}+n'$ and $\operatorname{Id}_{L"}+n''$ are both invertible.
Hence we find $\operatorname{Ker}(\xi+\eta)=L''$ and  $\operatorname{Im}(\xi+\eta)=L''$. That implies  $M_{\eta+\xi}=0$.
\end{proof}

\begin{lemma}\label{free} If $M\in\cO$ is free over $U(\fg_{-1})$, then $X_M\subset\fg_1$ and therefore $S\in\cS(M)$ implies $S\subset \Delta_{\ob}^+$.
\end{lemma}
\begin{proof} If $x\in \fg_{-1}$, then $M_x=0$ since $M$ is free over $\mathbb C[x]$. Let $x\in X$. Then $x$ can be written uniquely as $x^++x^-$ with 
$x^{\pm}\in\fg_{\pm 1}$. We claim that if $x^-\neq 0$, then $M_x=0$. Indeed, we apply Lemma~\ref{technical} with $u=z$, where~$z$ is introduced in equation~\eqref{gradel}, $\xi=x^+$ and $\eta=x^-$ and use the fact
that $M_{\eta}=0$.  
\end{proof}

If $\alpha$ is a root of $\fg$, we denote by $X_{\alpha}$ some non-zero element from the root space $\fg_\alpha$ and set $H_{\alpha}:=[X_\alpha,X_{-\alpha}]$.

\begin{lemma}\label{hereditary} Let $M$ be from the category~$\cO$, and $S\in\cS(M)$. Let $S$ be a disjoint union of two subsets $S_{1}$ and $S_{-1}$ and 
$h\in\fh^*$ be an element of the Cartan subalgebra, non-negative on all even positive roots. Assume that $\alpha(h)=i$ for all $\alpha\in S_i$ where~$i=\pm 1$.
Then $S_{-1}\in\cS(M)$.
\end{lemma}
\begin{proof} Follows again from Lemma~\ref{technical}. We write $X=X^++X^-$, where~$X^\pm =\sum_{\alpha\in S_{\pm 1}}X_\alpha$, set $\xi=X^+,\eta=X^-, u=h$.
\end{proof}

\begin{remark} More generally, it seems plausible that if $S\in\cS(M)$, then any subset $S'\subset S$ is also in $\cS(M)$. 
\end{remark}

\begin{lemma}
\label{reducegl}
Let $S=\{\varepsilon_{i_s}-\delta_{j_s}\,|\, s\in[1,k]\}$ be a set of $k$ mutually orthogonal positive odd roots. Set $a:=\min \{i_s\,|\, s\in[1,k]\}$ and $b:=\max\{ j_s\,|\, s\in [1,k]\}$ and let $\fg'\simeq\mathfrak{gl}(m-a+1|b)$ be the subalgebra of $\fg\simeq\mathfrak{gl}(m|n)$ generated by 
$X_{\pm(\varepsilon_i-\delta_j)}$ with $i\ge a$ and $j\le b$. Let $\lambda'$ be the restriction of $\lambda$ to the Cartan subalgebra of $\fg'$ and $\cS'$ be the set of subsets of mutually orthogonal linearly independent odd roots in $\fg'$ (clearly, $\cS'\subset\cS$). Then we have
$\cS(\lambda')=\cS(\lambda)\cap \cS'$.
\end{lemma}
\begin{proof} Let $S\in\cS'$ and $x=x_S$.
The Verma module $M(\lambda)$ is isomorphic to $M(\lambda')\otimes S(\fg/(\fg'+\fb))$ as a $\fg'$-module and therefore as a $\mathbb C[x]$-module, where we considered adjoint 
action on $\fg/(\fg'+\fb)$. In particular $M(\lambda')$ is a direct summand in $M(\lambda)$. Therefore $M(\lambda')_x\neq 0$ implies $M(\lambda)_x\neq 0$.
On the other hand, if $M(\lambda')_x=0$, then $M(\lambda)_x= 0$ by Lemma~\ref{general}(1). 
\end{proof}

\begin{corollary}\label{reductyp}
Consider the set $\{\varepsilon_{i_s}-\delta_{j_s}\}$ of all positive atypical roots for $\lambda\in\fh^\ast$. 
For a set $S$ of mutually orthogonal odd positive roots of the form $\varepsilon_i-\delta_j$ with $i> i_s$ and $j< j_s$ for every~$s$, we have $S\not\in \cS(\lambda)$.
\end{corollary}
\begin{proof}
This follows from Lemma~\ref{reducegl}, since $\lambda'$ is a typical weight.
\end{proof}

\begin{lemma}
\label{nointeg}
Consider $S=\{\alpha_1,\cdots,\alpha_k\}\in \cS$.  If $S\in \cS(\lambda)$, then  $( \lambda,\alpha_i)\in\mathbb Z$ for all $i=1,\dots,k$.
\end{lemma}
\begin{proof}
Let $h_i\in [\fg_{\alpha_i},\fg_{-\alpha_i}]$, such that $\beta(h_i)=(\beta,\alpha_i)$ for any weight $\beta$. For any $t_1,\dots,t_k\in \mathbb C\setminus 0$ consider
the  $\mathfrak{sl}(1|1)$-triple $\{x_S,y,h\}$, where~$h=t_1h_1+\dots+t_kh_k$. By assumtion, $S\in \cS(\lambda)$, so $M(\lambda)$ cannot be a typical module for the $\mathfrak{sl}(1|1)$-triple and hence there exists a weight $\mu$
such that $\mu(h)=0$. Since $\lambda-\mu$ is an integral linear combination of roots, we have  
$$\lambda(h)=\sum_{i=1}^kt_i(\lambda,\alpha_i)\in\sum_{i=1}^k\mathbb Z t_i.$$ 
For generic choice of $t_1,\dots,t_k$ this implies $(\lambda,\alpha_i)\in\mathbb Z$ for all $i=1,\dots,k$.
\end{proof}

\begin{lemma}\label{dominant}  Let $\fg=\mathfrak{gl}(n|n)$ and $\lambda\in P_0^{++}$ of degree of atypicality $n$. 
Then $X_{M(\lambda)}=\fg_1$.
\end{lemma}
\begin{proof} We consider the $\Z$-grading $\fg=\fg_{-1}\oplus\fg_0\oplus\fg_1$. By Lemma~\ref{free} it suffices to prove $\fg_1\subset X_{M(\lambda)}$. 
Note that every $M\in \cO$ can be equipped with the $\mathbb Z$-grading induced by the action of $z$ in equation \eqref{gradel}, let $M^+$ denote the highest degree component. 
As for any $x\in\fg_1$ we have $M(\lambda)^+\subseteq \ker x$, it would suffice to prove $M(\lambda)^+\not\subseteq xM(\lambda)$. We will prove this 
by using the property $L(\lambda)^+\not\subseteq xL(\lambda)$ for any $x\in \fg_1$, which is known to be true by Section 10 in~\cite{Duflo}.

Consider the exact sequence $M(\lambda)\xrightarrow{\pi} L(\lambda)\to 0$ of graded $\fg$-modules. Now assume that $M(\lambda)^+\subseteq xM(\lambda)$, so any 
$a\in M(\lambda)^+$ can be written as $a=xb$ for some $b\in M(\lambda)$. Then clearly $\pi(a)=x\pi(b)$ and, as the graded map $\pi$ is surjective, we find a 
contradiction with the fact that $L(\lambda)^+\subseteq xL(\lambda)$.
\end{proof}

Now we focus on the case $|S|=1$.

\begin{lemma}\label{oneroot} Let $\alpha=\varepsilon_i-\delta_j$ be a positive odd root. If $(\lambda+\rho,\alpha)=0$, then $\{\alpha\}\in\cS(\lambda)$.
\end{lemma}
\begin{proof} Consider the subset $\Gamma$ of odd positive roots defined by
$$\Gamma=\{\varepsilon_p-\delta_q | p>i,q\leq j,\,\,\text{or}\,\, p\geq i,q<j\}.$$
We set $\Sigma_\Gamma=\sum_{\gamma\in\Gamma}\gamma$, then we have $(\Sigma_\Gamma,\alpha)=-(\rho,\alpha)$.

Let $v\in M(\lambda)$ be a highest weight vector and
$$w:=\prod_{\gamma\in\Gamma} X_{-\gamma}v.$$
Then a quick check yields $X_{\alpha}w=0$.

We claim that $w\notin\im X_{\alpha}$. Indeed, $w$ is a weight vector of weight $\displaystyle\mu= \lambda-\Sigma_{\Gamma}$.
Assume $w=X_{\alpha}w'$. Without loss of generality we may assume that $w'$ is a weight vector of weight $\mu-\alpha$. Note that the weight space
$M(\lambda)_{\mu-\alpha}$ is one-dimensional. Therefore we may assume that $w'$ is proportional to $X_{-\alpha}w$. On the other hand
$$X_{\alpha}X_{-\alpha}w=H_{\alpha}w=0$$
since $(\mu,\alpha)=(\lambda-\Sigma_\Gamma,\alpha)=(\lambda+\rho,\alpha)=0$.

Therefore $X_\alpha\in X_{M(\lambda)}$ and the lemma is proven. 
\end{proof}

\begin{lemma}
\label{oneroot2} Let $\gamma=\varepsilon_p-\delta_q$ be a positive odd root with $(\lambda+\rho,\gamma)=0$. Then for $\alpha=\varepsilon_i-\delta_j$ with $i\le p$, $j\ge q$, $(\lambda+\rho,\varepsilon_i-\varepsilon_p)\in\Z_{\ge 0}$ and $(\lambda+\rho,\delta_q-\delta_j)\in\Z_{\le 0}$ we have $\{\alpha\}\in\cS(\lambda)$.
\end{lemma}
\begin{proof}
If $i=p$ and $j=q$, this is Lemma~\ref{oneroot}. Otherwise we set $\beta=\varepsilon_i-\varepsilon_p$, $\beta'=\delta_q-\delta_j$ and $r=r_\beta r_\beta'$ (or $r=r_\beta$ if $\beta'=0$ or $r=r_{\beta'}$ if $\beta=0$). By assumption and application of Verma's theorem in Theorem~4.6 in~\cite{MR2428237}, we have an embedding $M(r\cdot\lambda)\subset M(\lambda)$.

Since $(r\cdot\lambda+\rho,\varepsilon_i-\delta_j)=0$ we can repeat the construction in the proof of Lemma~\ref{oneroot} of a vector $w$ for the highest weight vector~$v$ of $M(r\cdot\lambda)$. We again have $X_\alpha w=0$. Suppose that $w=X_\alpha w'$. We may assume that
$w'$ has weight $\nu=\mu-\alpha$, where~$\mu$ is the weight of $w$. Although the corresponding weight space $M(\lambda)_{\nu}$ is not one dimensional, any vector
in this space is of form $$X_{-\alpha}\prod_{\gamma\in\Gamma}X_{-\gamma} u$$ for some $u\in U(\fg_{\oa})v$ of weight $r_{\beta}\cdot\lambda$. Therefore we have
$$X_{\alpha}X_{-\alpha}\prod_{\gamma\in\Gamma}X_{-\gamma} u\,\,\in\,\, H_\alpha \prod_{\gamma\in\Gamma}X_{-\gamma} u+\im X_{-\alpha}.$$   
But $(\mu,\alpha)=0$, hence $\displaystyle H_\alpha \prod_{\gamma\in\Gamma}X_{-\gamma} u=0$. Since $w\notin \im X_{-\alpha}$, we obtain that
$X_{\alpha}w'$ is never $w$. Thus, $M(\lambda)_{X_\alpha}\neq 0$. 
\end{proof}

\begin{proposition}\label{typicalverma} Let $M$ be a Verma or simple module. Then $X_M=0$ if and only if the highest weight of $M$ is typical.
\end{proposition}
\begin{proof}
The case where~$M$ is a Verma module is an immediate consequence of Lemma~\ref{oneroot}, so we consider $M\simeq L(\lambda)$ simple for $\lambda\in\fh^\ast$. If $\lambda$ is typical, the claim follows from Theorem~\ref{finpdind}, so we are left with $\lambda$ atypical.

 Consider $(\lambda+\rho,\gamma)=0$ for $\gamma$ an odd root. If $\gamma$ is simple the result follows immediately since $X_{-\gamma}v=0$ for a highest weight vector $v$ (since $U(\fn^+)X_{-\gamma}v=0$) while $v\not\in\im X_{-\gamma}$. If $\gamma$ is not simple we can consider a sequence of odd reflections (see Section 3.5 in~\cite{bookMusson}) to obtain a system of positive roots in which $\gamma$ is simple. If one of these odd reflections is atypical for the simple module we can take the corresponding root as $\gamma$, so we consider the situation where each odd reflection is typical. In the new system of roots (with corresponding half-sum $\widetilde{\rho}$) the simple module will have highest weight $\widetilde{\lambda}=\lambda+\rho-\widetilde{\rho}$, which thus satisfies $( \widetilde{\lambda}+\widetilde{\rho},\gamma)=0$, so we end up in the setting where~$\gamma$ is simple.
\end{proof}

\subsection{The cases $\fg=\mathfrak{gl}(1|n)$ and $\fg=\mathfrak{gl}(m|1)$}
For these cases we determine $\cS(\lambda)$, or equivalently $X_{M(\lambda)}$, for any $\lambda\in\fh^\ast$.
\begin{theorem}\label{verma} Let $\fg=\mathfrak{gl}(1|n)$ and $\lambda$ be some atypical weight. Let $p\leq n$ be such that
$(\lambda+\rho, \varepsilon_1-\delta_p)=0$ and $(\lambda+\rho,\varepsilon_1-\delta_j)\neq 0$ for all $j<p$. Then $\{\varepsilon_1-\delta_i\}\in \cS(\lambda)$
if and only if $i\geq p$ and $(\lambda+\rho,\varepsilon_1-\delta_i)\in\mathbb Z_{\leq 0}$. 
\end{theorem}
\begin{proof} If $\alpha=\varepsilon_1-\delta_i$ for some $i<p$, then $\{\alpha\}\notin\cS(\lambda)$, by Corollary~\ref{reductyp}. If $\alpha=\varepsilon_1-\delta_i$ satisfies $i\geq p$ and $(\lambda+\rho,\varepsilon_1-\delta_i)\in\mathbb Z_{\leq 0}$, then $\{\alpha\}\in\cS(\lambda)$, by Lemma~\ref{oneroot2}. 

Now let us assume that $\alpha=\varepsilon_1-\delta_i$ for some $i>p$, but $(\lambda+\rho,\alpha)\notin\mathbb Z_{\leq 0}$. Then we have either
$(\lambda,\alpha)\notin\mathbb Z$ or $(\lambda,\alpha)\in\Z_{\ge i}$. The first case is covered by Lemma~\ref{nointeg}. For the second case, we use Lemma~\ref{reducegl} with $S=\{\varepsilon_1-\delta_p\}$ and hence $\fg'\simeq \mathfrak{gl}(1|p)$. It thus suffices to consider the Verma module $M(\lambda')$ of $\fg'$ and to show that
$M(\lambda')_{X_\alpha}=0$. Note that $(\beta,\alpha)\geq 0$ for any negative even root $\beta$ of $\fg'$ and $(\beta,\alpha)=-1$ 
for any negative odd root $\beta\neq -\alpha$ of $\fg'$. Therefore
$(\mu,\alpha)\neq 0$ for any weight $\mu$ of $M(\lambda')$. Therefore $M(\lambda')$ as a module over the $\mathfrak{sl}(1|1)$-subalgebra, generated by $X_{\pm \alpha}$,  
is a direct sum of typical modules. Thus, $M(\lambda')_{X_\alpha}=0$.
\end{proof}

\begin{remark} The above theorem implies that in contrast with the finite dimensional case, see Lemma 2.1 in \cite{Duflo}, there are $M\in\cO$ for which the associated variety $X_M$ is not closed. 
For example, if $\fg=\mathfrak{gl}(1|2)$ and $\lambda=3\delta_2$, then $X_{M(\lambda)}=\fg_1\setminus \mathbb C (\varepsilon_1-\delta_2)$ is not closed.
\end{remark}
The isomorphism $\mathfrak{gl}(1|n)\simeq \mathfrak{gl}(n|1)$ links the highest weight structure of category~$\cO$ with distinguished system of positive roots for $\mathfrak{gl}(1|n)$ to category~$\cO$ with anti-distinguished system of positive roots for $\mathfrak{gl}(n|1)$. The following result is therefore not identical to Theorem \ref{verma}, but can be proved similarly.
\begin{theorem}\label{verma2}
Let $\fg=\mathfrak{gl}(m|1)$ and $\lambda$ be some atypical weight. Let $p\leq m$ be such that
$(\lambda+\rho, \varepsilon_p-\delta_1)=0$ and $(\lambda+\rho,\varepsilon_j-\delta_1)\neq 0$ for all $j>p$. Then $\{\varepsilon_i-\delta_1\}\in \cS(\lambda)$
if and only if $j\le p$ and $(\lambda+\rho,\varepsilon_i-\delta_1)\in\mathbb Z_{\ge 0}$. 
\end{theorem}

\subsection{The case $\fg=\mathfrak{gl}(2|2)$} In this case we have four positive odd roots
$$\alpha=\varepsilon_2-\delta_1, \beta=\varepsilon_1-\delta_1,\gamma=\varepsilon_2-\delta_2, \delta=\varepsilon_1-\delta_2.$$
We represent weights by using the bijection $\fh^\ast\simeq \C^{2|2}$, as a natural extension of equation~\eqref{hat} and determine all $\cS(\lambda)$ for $\lambda\in\fh^\ast$.

\begin{lemma}\label{atypicality1} Let $\lambda$ be a weight with degree of atypicality $1$. Then up to the shift by the weight $(t,t|t,t)$, for any $t\in\C$,
we have the following options.
\begin{enumerate}
\item $\mu^\lambda=(0,a|b,0)$ with $a,b\in\mC$ such that $a\not=b$ and $ab\neq 0$. Then
$\cS(\lambda)=\{\{\delta\},\emptyset\}$. 
\item $\mu^\lambda=(0,a|0,b)$ with $a,b\in\mC$ such that $a\not=b$ and $a\neq 0$. If 
$b\notin \mathbb Z_{\geq 0}$, then $\cS(\lambda)=\{\{\beta\},\emptyset\}$. If 
$b\in \mathbb Z_{\geq 0}$, then $\cS(\lambda)=\{\{\beta\},\{\delta\},\emptyset\}$. 
\item $\mu^\lambda=(a,0|b,0)$ with $a,b\in\mC$ such that $a\not=b$ and $b\geq 0$. If 
$a\notin \mathbb Z_{\geq 0}$, then $\cS(\lambda)=\{\{\gamma\},\emptyset\}$. If 
$a\in \mathbb Z_{\geq 0}$, then $\cS(\lambda)=\{\{\gamma\},\{\delta\},\emptyset\}$. 
\item $\mu^\lambda=(a,0|0,b)$ with $a,b\in\mC$ such that $a\not=b$. If 
$b\notin \mathbb Z_{\geq 0}$ and $a\notin \mathbb Z_{\geq 0}$, then $\cS(\lambda)=\{\{\alpha\},\emptyset\}$. If 
$b\in \mathbb Z_{\geq 0}$ but $a\notin \mathbb Z_{\geq 0}$, then $\cS(\lambda)=\{\{\gamma\},\{\alpha\},\emptyset\}$. 
If $a\in \mathbb Z_{\geq 0}$ but
$b\notin \mathbb Z_{\geq 0}$, then $\cS(\lambda)=\{\{\beta\},\{\alpha\},\emptyset\}$. Finally, if $a\in \mathbb Z_{\geq 0}$ and $b\in \mathbb Z_{\geq 0}$, then 
$\cS(\lambda)=\{\{\delta\},\{\gamma\},\{\beta\},\{\alpha\},\emptyset\}$.
\end{enumerate}
\end{lemma}
\begin{proof} The proof of this lemma is a straightforward application of Lemma~\ref{oneroot2} and reduction to the case of $\mathfrak{gl}(1|2)$. We leave it
as an exercise to the reader. 
\end{proof}
\begin{lemma}\label{atypicality2} Let $\lambda$ be a weight with degree of atypicality $2$ and $A(\lambda)$ denote the set of all odd positive roots atypical to 
$\lambda$.
\begin{enumerate} 
\item If $\lambda$ is regular dominant integral, then $\cS(\lambda)$ is the set of all subsets of mutually orthogonal roots in $\Delta_1^+$.
\item If $\lambda$ is regular integral and neither dominant nor anti-dominant, then 
$\cS(\lambda)=\{\{\delta\},\{\beta,\gamma\},\{\beta\},\{\gamma\},\emptyset\}$.
\item If $\lambda=-\rho$, then $\cS(\lambda)$ is the set of all subsets of mutually orthogonal roots in $\Delta_1^+=A(-\rho)$.
\item If $\lambda$ is regular, non-integral or anti-dominant integral,  then $\cS(\lambda)$ is the set of all subsets of $A(\lambda)$.
\end{enumerate}
\end{lemma}
\begin{proof} We first observe that (1) is a particular case of Lemma~\ref{dominant}.

Next, we prove (2). We assume that $\mu^\lambda=(a,0|a,0)$ with positive integral $a$, the case $(0,a|0,a)$ being similar. Then $\cS(\lambda)$ contains $\{\beta\}$
and and $\{\gamma\}$ by Lemma~\ref{oneroot} and $\{\delta\}$ by Lemma~\ref{oneroot2}. On the other hand, $\{\alpha\}\notin \cS(\lambda)$ by Lemma~\ref{reducegl}.
Also we can apply Lemma~\ref{hereditary} to $S=\{\alpha,\delta\}$ with $h=\varepsilon_1-\varepsilon_2$ and conclude that 
$\{\alpha,\delta\}\notin \cS(\lambda)$.    It remains to show that $\{\beta,\gamma\}\in \cS(\lambda)$. 
For this we take $w=X_{-\alpha}v$, where~$v$ denotes the highest weight vector, and let
$x=X_{\beta}+X_{\gamma}$. Then $xw=0$ and we claim that $w\notin\im x$ by the same argument in the proof of Lemma~\ref{oneroot2} since we have
$$X_{\beta}X_{-\beta}X_{-\alpha}v=(\lambda-\alpha,\beta)X_{-\alpha}v=0,\quad X_{\gamma}X_{-\gamma}X_{-\alpha}v=(\lambda-\alpha,\gamma)X_{-\alpha}v=0.$$

Let us prove (3) now. As $\mu^\lambda=(0,0|0,0)$, Lemma~\ref{oneroot} implies that $\cS(\lambda)$ contains all singletons. 
Furthermore, $\{\beta,\gamma\}\in \cS(\lambda)$ by the same 
argument as above. To prove that $\{\alpha,\delta\}\in \cS(\lambda)$ set $x=X_{\alpha}+X_{\delta}$. Let $M$ denote the projection of
$M(0,-1|0,0)\otimes U$ on the most atypical block. Lemma~\ref{general} (1),(2) implies that  $M_x=0$. On the other hand, $M$ has a filtration by three Verma modules,
$M(0,0|0,0)$, $M(0,-1|-1,0)$ and $M(0,-1|0,-1)$. From the previous cases we have $M(0,-1|-1,0)_x\neq 0$ and  $M(0,-1|0,-1)_x=0$. Thus, we must have
$M(0,0|0,0)_x\neq 0$.

Finally, let us deal with (4).
Here we have several subcases to consider.

If $A(\lambda)=\{\beta,\gamma\}$, then we may assume $\lambda=(a,0|a,0)$ with non-integral $a$. 
Any subset of $A(\lambda)$ is in $\cS(\lambda)$ by the same argument as 
in the previous case. Furthemore, $\{\alpha\},\{\delta\}\notin \cS(\lambda)$ by Lemma~\ref{nointeg}, and 
$\{\alpha,\delta\}\notin \cS(\lambda)$ by Lemma~\ref{hereditary}. 

If $A(\lambda)=\{\alpha,\delta\}$, then we may assume
$\mu^\lambda=(-a,0|0,-a)$ with $a\in \mathbb Z_{> 0}$. Then $\{\beta\},\{\gamma\}\notin \cS(\lambda)$ by Lemma
\ref{reductyp}. Moreover, Lemma~\ref{hereditary} implies that  $\{\beta,\gamma\}\notin \cS(\lambda)$.
On the other hand, $\{\alpha\},\{\delta\}\in\cS(\lambda)$ by Lemma~\ref{oneroot}. Finally, to prove that  $\{\alpha,\delta\}\in \cS(\lambda)$ we use the same trick
with translation functor as in (3). More precisely, we set again  $x=X_{\alpha}+X_{\delta}$ and consider the projection $M$ of $M(-a-1,0|0,-a)\otimes U$ 
on the most atypical block. Now $M$ is filtred by two Verma modules:  $M(-a-1,0|0,-a-1)$and $M(-a,0|0,-a)$. Using the result for $a=0$ we obtain by induction
in $a$ that $M(-a,0|0,-a)_x\neq 0$.   
\end{proof}


\section{Projective dimensions and blocks of category~$\cO$}
\label{secblocks}
In this entire section we consider $\fg=\mathfrak{gl}(m|n)$ or $\fg=\mathfrak{sl}(m|n)$ with distinguished Borel subalgebra. We denote by $\mathtt{a}:W\to\mathbb{N}$ Lusztig's $\mathtt{a}$-function, see \cite{Lu}.

\subsection{Projective dimensions of structural modules}

\begin{theorem}
\label{pdstruct}
We have the following connection between projective dimensions of structural modules in the categories $\cO$ and $\cO^0$, for $\lambda\in\fh^\ast$:
\begin{enumerate}[$($i$)$]
\item $\pd_{\cO}L(\lambda)=\pd_{\cO^0}L_0(\lambda)$ if $\lambda$ is typical, otherwise $\pd_{\cO}L(\lambda)=\infty$;
\item $\pd_{\cO}M(\lambda)=\pd_{\cO^0}M_0(\lambda)$ if $\lambda$ is typical, otherwise $\pd_{\cO}M(\lambda)=\infty$;
\item $\pd_{\cO}I(\lambda)=\pd_{\cO^0}I_0(\lambda)$.
\end{enumerate}
This implies that for $\lambda\in P_0$, we have 
$$\pd_{\cO}I(\lambda)=\mathtt{a}(w_0x_\lambda)$$
with $x_\lambda$ the longest Weyl group element such that $x^{-1}_\lambda\cdot\lambda$ is dominant.
\end{theorem}
\begin{proof}
Properties $(i)$ and $(ii)$ for typical $\lambda$ follow from the fact that Brundan's KL polynomials for typical weights correspond to those for $\mathfrak{g}_0$, see e.g. Lemma~\ref{KLorbit}. Properties $(i)$ and $(ii)$ for atypical~$\lambda$ follow from the combination of Proposition~\ref{typicalverma} and Theorem \ref{finpdind} $(i)\leftrightarrow (ii)$.

According to Lemma~\ref{propind}$(i)$ and $(ii)$, to prove $(iii)$ it suffices to prove that $I_0(\lambda)$ is a direct summand of $\res^{\fg}_{\fg_0} I(\widetilde\lambda)$, with $\widetilde\lambda:=\lambda+2\rho_1$ (as $\pd I_0(\widetilde\lambda)=\pd I_0(\lambda)$) and that $I(\lambda)$ is a direct summand of $\ind^{\fg}_{\fg_0}I_0(\lambda)$. Both the induced and restricted module are injective and $\ind^{\fg}_{\fg_0}\simeq \coind^{\fg}_{\fg_0}$, so the properties
$$\Hom_{\fg_0}(L_0(\lambda), \res^{\fg}_{\fg_0} I(\widetilde\lambda))\simeq \Hom_{\fg}(\ind^{\fg}_{\fg_0} L_0(\lambda), I(\widetilde\lambda))=[\ind^{\fg}_{\fg_0} L_0(\lambda): L(\widetilde\lambda)]\not=0\,\,\,\mbox{and} $$
$$\Hom_{\fg}(L(\lambda), \ind^{\fg}_{\fg_0} I_0(\lambda))\simeq \Hom_{\fg_0}(\res^{\fg}_{\fg_0} L(\lambda), I_0(\lambda))=[\res^{\fg}_{\fg_0} L(\lambda): L_0(\lambda)]\not=0 $$
conclude the proof of $(iii)$.

Finally, the projective dimension of $I(\lambda)$ follows from $(iii)$ and Theorem 16 of~\cite{MR2366357}, see also Theorem 44(ii) in~\cite{SHPO4}.
\end{proof}

\begin{remark}
The theorem reduces the question concerning projective dimension of simple and Verma modules to the corresponding questions for Lie algebras. For integral regular weights these are well-known, see \cite{MR2366357}. For singular blocks only estimates and special cases are known at the moment, see \cite{SHPO4, CouMaz2}.
\end{remark}

Theorem 44(ii) in \cite{SHPO4} and Theorem \ref{pdstruct} lead to the following immediate consequences.
\begin{corollary}
\label{pdextr}
Consider $\lambda\in\fh^\ast$, then
\begin{enumerate}[$($i$)$]
\item $\dim L(\lambda) <\infty \quad\Leftrightarrow\quad \pd_{\cO} I(\lambda)=2l(w_0);$
\item $\lambda$ is anti-dominant $\quad\Leftrightarrow\quad  \pd_{\cO} I(\lambda)=0.$
\end{enumerate}
\end{corollary}

Property $(ii)$ was first obtained through other methods in Theorem 2.22 of~\cite{BLW}.

\subsection{Finitistic global dimension of blocks}

\begin{theorem}
\label{fdblock}
The finitistic global homological dimension of the block $\cO_\xi$ for an integral linkage class $\xi$ is given by
$$\fd\cO_\xi=2\mathtt{a}(w_0w_0^\xi).$$
\end{theorem}
\begin{proof}
The proof of Theorem 3 in~\cite{preprint} implies that $\fd\cO_\xi$ for a classical Lie superalgebra is equal to the highest projective dimension of an injective module. The result therefore follows immediately from Theorem 26(ii) in \cite{CouMaz2} and Theorem \ref{pdstruct}.
\end{proof}

\subsection{Blocks in category~$\cO$}
In this subsection we demonstrate a principle which is responsible for the fact that different integral blocks in category~$\cO$ will almost never be equivalent, even when they have the same degree of atypicality and singularity of core.

The origin of this new phenomenon is that each atypical integral block contains simple objects which are more regular or more singular. The category behaves differently `around' these objects. The `distance' between these objects in the category is determined by how far the core is distanced from the walls of the Weyl chamber. We make this explicit for blocks for $\mathfrak{sl}(3|1)$ with regular core, by using our results on projective dimensions.

In order to avoid the obvious equivalences of blocks coming from the centre~$\mathfrak{z}(\fg)$, see Lemma 3.5 in~\cite{CMW}, we consider $\fg=\mathfrak{sl}(3|1)$ rather than $\mathfrak{gl}(3|1)$. We use the fact that an equivalence of abelian categories is always given in terms of exact functors.

\begin{theorem}
\label{thmsl31}
Consider $\fg=\mathfrak{sl}(3|1)$. No two atypical integral blocks $\cO_{\xi}$ with regular core $\chi'_\xi$ are equivalent.
\end{theorem}
\begin{proof}
We use the notation $\Lambda=\mZ^{3|1}$ of $\mathfrak{gl}(3|1)$-weights, silently making the relevant identification.
The integral linkage classes with regular core are given by $\xi^p=[(p,1,0|0)]$, parametrised by $p\in\N$ with $p>1$. We label the set $\Lambda^{++}\cap\xi^p$ as $\{\lambda^p_i\,|\, i\in\Z\}$ with
\begin{itemize}
\item $\lambda^p_{i}= (p,1,i|i)$ for $i\le 0$;
\item $\lambda^p_i=(p,i+1,1|i+1)$ for $0< i <p-1$;
\item $\lambda^p_i=(i+2,p,1|i+2)$ for $i\ge p-1$.
\end{itemize}

From Corollary \ref{pdextr} we know that finite dimensional and anti-dominant simple modules are categorically defined. Any equivalence of categories between $\cO_{\xi^p}$ and $\cO_{\xi^{p'}}$ must therefore preserve these two types of simple modules. 

We will construct a categorical invariant in the form of  a graph. This graph is given by the $\Ext^1$-quiver of the subcategory of finite dimensional modules in $\cO_\xi$, where in each node $\lambda\in \Lambda^{++}$ we write the number of anti-dominant simple subquotients in $P(\lambda)$. This number is denoted by $\flat\lambda$.

The subsequent Lemma~\ref{Vermaanti} implies that the number of anti-dominant simple subquotients in $P(\lambda)$ is equal to twice the number of Verma modules in its standard filtration. By equation~\eqref{resCheng}, this means
$$\flat\lambda \;=\; 2\sum_{\nu\in\Lambda}d_{\lambda,\nu}(1).$$

Computing $d_{\lambda,\nu}$ is a direct application of the bumping procedure, see Example 3.3 in~\cite{Brundan3}. It follows that $\sum_{\nu\in\Lambda}d_{\lambda,\nu}(1)-1$ is equal to the minimal strictly positive number of times the (unique) atypical positive root of $\lambda$ can be added to $\lambda$ such that the result is another regular weight. 

The $\Ext^1$-quiver of the category of finite dimensional weight modules of $\mathfrak{sl}(3|1)$ is well-known to be of Dynkin type~$A_\infty$, which follows e.g. from the penultimate paragraph in Appendix~\ref{gl21} and Theorem 2 of \cite{MR2734963}. The combination of these results yields the following graphs: for $p\ge 3$ we have 
\begin{displaymath}
    \xymatrix{
        \cdots \bullet_{\flat\lambda^p_{-2}=4}\ar[r] &\bullet_{\flat\lambda^p_{-1}=4}\ar[r]\ar[l] &\bullet_{\flat\lambda^p_0=6}\ar[r]\ar[l] & \bullet_{\flat\lambda^p_{1}=4} \ar[r]\ar[l]&  \bullet_{\flat\lambda^p_{2}=4}\ar[r]\ar[l] & \cdots\ar[l]   }
\end{displaymath}
\begin{displaymath}
    \xymatrix{
        \cdots \bullet_{\flat\lambda^p_{p-4}=4}\ar[r] &\bullet_{\flat\lambda^p_{p-3}=4}\ar[r]\ar[l] &\bullet_{\flat\lambda^p_{p-2}=6}\ar[r]\ar[l] & \bullet_{\flat\lambda^p_{p-1}=4} \ar[r]\ar[l]&  \bullet_{\flat\lambda^p_{p}=4}\ar[r]\ar[l] & \cdots\ar[l]   },
\end{displaymath}
meaning $p-3$ nodes between the two exceptional nodes and if $p=2$ we have
\begin{displaymath}
    \xymatrix{
        \cdots\bullet_{\flat\lambda^2_{-2}=4}\ar[r] &\bullet_{\flat\lambda^2_{-1}=4}\ar[r]\ar[l] & \bullet_{\flat\lambda^2_{0}=8} \ar[r]\ar[l]& \bullet_{\flat \lambda_{1}^2=4}\ar[r]\ar[l] & \bullet_{\flat\lambda^2_{2}=4}\ar[r]\ar[l] & \bullet\ar[l]\cdots  }.
\end{displaymath}
Since each diagram is different from the others, the result follows.
\end{proof}

\begin{corollary}
Category~$\cO_{\Z}$ for a basic classical Lie superalgebra can contain infinitely many nonequivalent blocks.
\end{corollary}

\begin{remark}
An alternative categorical invariant to $\flat \lambda$ would be to take $[P(\lambda):L(\lambda)]$. By BGG reciprocity and the fact that in the $\mathfrak{sl}(3|1)$ case the standard filtration of $P(\lambda)$ is multiplicity free (which follows from computation), $[P(\lambda):L(\lambda)]$ corresponds to the number of Verma modules in the standard filtration of $P(\lambda)$, so $\flat\lambda=2[P(\lambda):L(\lambda)]$.
\end{remark}
\begin{lemma}
\label{Vermaanti}
Any atypical integral Verma module of $\mathfrak{sl}(3|1)$ contains exactly two simple subquotients which have an anti-dominant highest weight.
\end{lemma}
\begin{proof}
Take $\mu\in P_0$. Any Verma module $M(\mu)\simeq U(\fg)\otimes_{U(\fg_0+\fg_1)} M_0(\mu)$ has a Kac filtration. Moreover, the number of times $K(\lambda)$ appears in this filtration is equal to $[M_0(\mu): L_0(\lambda)]$. Since $\Lambda\fg_{-1}$ is a finite dimensional $\fg_0$-module, the only possibility for $K(\lambda)$ to have an anti-dominant simple subquotient is if $\lambda$ is anti-dominant itself. Now every Verma module $M_0(\mu)$ has exactly one anti-dominant simple subquotient, it hence suffices to prove the lemma for anti-dominant Verma modules $M(\mu)=K(\mu)$.

By construction $K(\mu)$ admits $4$ different eigenvalues of $z$ in equation~\eqref{gradel}. Lemma~6.10(i) in~\cite{CouMus} implies that any anti-dominant atypical simple module for $\mathfrak{sl}(3|1)$ admits $3$ different such eigenvalues. This implies that any anti-dominant simple subquotient $L(\nu)$ of $K(\mu)$ with $\nu\not=\mu$ must satisfy~$\nu\in W\cdot(\mu-\gamma)$ with $\gamma\in\Delta^+_1$ atypical for $\mu$. Note that $W\cdot(\mu-\gamma)$ does not depend on the atypical root $\gamma$, in case there is more than one. This hence leaves only one possibility besides $\mu$, {\it viz.} the unique anti-dominant weight in $W\cdot(\mu-\gamma)$, which we denote by $\nu$. We claim that $[K(\mu):L(\nu)]\le 1$. This follows from looking at weight spaces corresponding to the weight $\mu-2\rho_1$. In $K(\mu)$, this has dimension one, in $L(\mu)$ dimension zero and in $L(\nu)$ dimension one. The fact that there appears at least two anti-dominant simple modules follows from the fact that both the socle and top of $K(\mu)$ must be anti-dominant.
\end{proof}

\begin{remark}
Classically the equivalences between regular integral blocks are often given by translation functors, see Section 7.8 in~\cite{MR2428237}. The same holds for finite dimensional modules for Lie superalgebras. It is interesting to note how this fails for category~$\cO$ for superalgebras. The functor $F_p$ maps the block corresponding to $(p,1,0|0)$ to the one corresponding to $(p+1,1,0|0)$. According to Theorem 2.4 in~\cite{Ku}, this maps every simple module to a simple module with only three exceptions: $F_p$ maps the singular modules $L(p,p,1|p)$, $L(p,1,p|p)$ and $L(1,p,p|p)$ to indecomposable modules which are not simple.  

The principal used in the proof of Theorem \ref{thmsl31}, that the singular objects appear at different positions in the two blocks, is also responsible for the problematic behaviour of the translation functor. This principle namely causes translation onto the walls for some modules and translation out of the wall for other.
\end{remark}


\section{Complexity in category~$\cO$}
\label{seccomp}
In this section we introduce the notion of complexity in category~$\cO$ for basic classical Lie superalgebras, as the rate of polynomial growth of a minimal projective resolution of a module. We prove that this is well-defined, {\it i.e.} it is finite for every module. Then we study the relation between degree of atypicality and complexity of Verma and simple modules for $\mathfrak{gl}(m|n)$. Similar results for the category~$\cF$ have been obtained by Boe, Kujawa and Nakano in~\cite{BKN1, BKN2}.

\subsection{Definition and basic properties}

The usual notion of complexity, as introduced by Alperin, measures the rate of growth of the dimension in a minimal projective resolution. Since the projective objects in category~$\cO$ are infinite dimensional we need to consider instead the number of indecomposable projective objects. This variation has also been studied for the category of finite dimensional modules of $\mathfrak{gl}(m|n)$ in Section 9 in~\cite{BKN2} and is (contrary to the original approach) a categorical invariant.

\begin{definition}\label{defC}
For $M\in\cO$ we define $c_{\cO}(M)$, the complexity of $M$ in category~$\cO$, as
$$c_{\cO}(M)=r\left(\sum_{\mu\in\fh^\ast}\dim\Ext^\bullet_{\cO}(M,L(\mu))\right).$$
The rate of growth $r(c^\bullet)$ of a sequence of numbers $c^\bullet$ is defined as the smallest non-negative integer $k$ such that there is a constant $C>0$ for which $ c^j\le C j^{k-1}$ for all $j>0$. In case the $c^j$ are not finite or no such integer exists, we set $r(c^\bullet)=\infty$.
\end{definition}
By definition, the complexity of a module is zero if and only if it has finite projective dimension. Immediate from the definition we have the following properties.
\begin{lemma}
\label{lemcompseq}
Consider a short exact sequence $A_1\hookrightarrow A_2\tto A_3$ in category~$\cO$, then
$$c_{\cO}(A_i)\le\max\{c_{\cO}(A_j),c_{\cO}(A_k)\}$$
for any permutation $\{i,j,k\}$ of $\{1,2,3\}$.
\end{lemma}

As main results of this subsection we prove that this notion of complexity is well-defined for category~$\cO$ for basic classical Lie superalgebras and that translation functors cannot increase complexity.
\begin{proposition}
For any $M\in\cO$, the value $c_{\cO}(M)$ is finite dimensional, more precisely
$$c_{\cO}(M)\le\dim\fg_{\ob}.$$
\end{proposition}
\begin{proof}
We prove this by induction on the (finite) projective dimension of $\Res M$ in category~$\cO^0$. Assume that the property holds for any $K\in\cO$ with $\pd_{\cO^0}\Res K< p$. Denote the projective cover of an $M\in\cO$, with $\pd_{\cO^0}\Res M= p$, by $P$ and the kernel of the morphism $P\tto M$ by $N$. Since $\pd_{\cO^0}\Res N< p$ and $c_{\cO}(P)=0$, the induction step and Lemma~\ref{lemcompseq} imply that $$c_{\cO}(M)\le c_{\cO}(N)\le \dim\fg_{\ob}.$$

It remains to be proved that $c_{\cO}(M)\le\dim\fg_{\ob}$ in case $\Res M$ is projective in $\cO^0$. We consider the Chevalley-Eilenberg resolution of $\mC$ for $(\fg,\fg_{\oa})$-relative homological algebra. It was proved explicitly in Proposition 2.4.1 of~\cite{BKN1} that this is a $(\fg,\fg_{\oa})$-projective resolution of $\mC$. Tensoring this resolution with $M$ yields an exact complex
$$\cdots\to U(\fg)\otimes_{U(\fg_{\oa})}(S^j(\fg_{\ob})\otimes \Res M) \to\cdots \to U(\fg)\otimes_{\fg_{\oa}}(\fg_{\ob}\otimes\Res M)\to M\to 0.$$
Since $\Res M$ is projective in $\cO^0$, this is a projective resolution in $\cO$ of $M$. Applying Frobenius reciprocity then implies
$$\sum_{\mu\in\fh^\ast}\dim \Ext^j_{\cO}(M,L(\mu))\le \sum_{\mu\in\fh^\ast}\dim\Hom_{\fg_{\oa}}(S^j(\fg_{\ob})\otimes \Res M, \Res L(\mu)).$$
Lemma~\ref{bound2} applied to $\fg_{\oa}$ and Lemma~\ref{bound3} then allow to conclude
$$\sum_{\mu\in\fh^\ast}\dim \Ext^j_{\cO}(M,L(\mu))\le C_{\fg_{\oa}}\,\widetilde{C}_{\fg}\,q\, \dim S^j(\fg_{\ob}) .$$
with $q$ the number of indecomposable projective modules of $\cO^0$ in $\Res M$. The result thus follows from the fact that the polynomial grow rate of $\dim S^j(\fg_{\ob})$ is $\dim\fg_{\ob}$.
\end{proof}

\begin{proposition}
\label{propT}
Consider any finite dimensional module $V$ and translation functor 
$$T^{\chi,\chi'}_V:\cO_{\chi}\to \cO_{\chi'}\,\,:\quad M\in\cO_{\chi}\mapsto T^{\chi,\chi'}_V(M)= (M\otimes V)_{\chi'}\in\cO_{\chi'}.$$
Then we have $c_{\cO}(T^{\chi,\chi'}_V(M))\le c_{\cO}(M)$.
\end{proposition}
\begin{proof}
Consider a minimal projective resolution of $M$. Its rate of polynomial growth is~$c_{\cO}(M)$. This projective resolution is mapped by the exact functor $T^{\chi,\chi'}_V$ to a (not necessarily minimal) projective resolution of $T^{\chi,\chi'}_V(M)$. The polynomial rate of that resolution is smaller or equal to $c_{\cO}(M)$, by Lemma~\ref{bound2}.
\end{proof}

\begin{corollary}
Consider a translation functor $T=T^{\chi,\chi'}_V$ with adjoint $\widetilde T=T^{\chi',\chi}_{V^\ast}$ and $M\in\cO_{\chi}$. If for $M':=TM$ we have $\widetilde T M'\simeq M$, then $c_{\cO}(M)=c_{\cO}(M').$
\end{corollary}

\subsection{Complexity of Verma modules for $\mathfrak{gl}(m|n)$}

By the equivalences of categories in \cite{CMW}, it suffices to compute the complexity for Verma modules in integral blocks.  Their complexity is in principle determined by Brundan's KL polynomials. Using equation~\eqref{defp}, for any $\lambda\in\Lambda$ we introduce the notation
\begin{equation}\label{KLVcompl}p_{\lambda}^j=\sum_{\nu\in\Lambda}\dim\Ext_{\cO}^j(M(\lambda),L(\nu))=\frac{1}{j!} \sum_{\nu\in\Lambda} \left(\frac{\partial^j}{\partial q^j}p_{\lambda,\nu}(q)\right)_{q=0}.\end{equation}

\begin{theorem}
\label{complVerma}
There are constants $C_{k}$, such that for any $\lambda\in\Lambda$ with $\sharp[\lambda]=k$ we have
\begin{equation}\label{eqthm}p_{\lambda}^j\;\le\;C_{k}\, j^{k-1},\qquad \forall j>0.
\end{equation}
The complexity of a Verma module satisfies $c_{\cO}(M(\lambda))=\sharp[\lambda]$ if $\lambda\in\Lambda$ is regular and $c_{\cO}(M(\lambda))\le\sharp[\lambda]$ if $\lambda\in\Lambda$ is singular.
\end{theorem}

\begin{theorem}
\label{complK}
Any module $M\in \cO_\xi$ which is either $\fg_1$-free or $\fg_{-1}$-free has $c_{\cO}(M)\le \sharp\xi$.
\end{theorem}

The remainder of this subsection is devoted to the proof of these theorems, but first we observe that the corresponding property for category~$\cF$ as derived in~\cite{BKN2} can be made even more precise.

\begin{lemma}
\label{lemF}
For any $\lambda\in\Lambda^{++}$ with $\sharp[\lambda]=k$, we have
$$\sum_{\nu\in\Lambda^{++}}\dim\Ext^j_{\cF}(K(\lambda),L(\nu))\,=\,\binom{k+j-1}{k-1}\,=\,\frac{1}{(k-1)!}(j^{k-1}+\frac{1}{2}k(k-1)j^{k-2}+\cdots).$$
\end{lemma}
\begin{proof}
Theorem 4.51 and Corollary 3.39(ii) in~\cite{Brundan} imply that
$$\sum_{\nu\in\Lambda^{++}}\dim\Ext^j_{\cF}(K(\lambda),L(\nu))\,=\,\sharp\{\theta\in \N^{k}\,|\,\, |\theta|=j \},$$
which proves the statement.
\end{proof}

For the subsequent proofs, we divide the atypical weights into four mutually exclusive types. For $\lambda\in\Lambda$, we set $a_\lambda$ equal to the highest label which appears on both sides.
\begin{enumerate}[(a)]
\item There is no label in $\lambda$ higher than $a_\lambda$.
\item There is a label in $\lambda$ higher than $a_\lambda$, but no label equal to $a_\lambda+1$.
\item There is a label equal to $a_\lambda+1$, but only one occurrence of $a_\lambda$ on each side.
\item There is a label equal to $a_\lambda+1$, as well as multiple occurrences of $a_\lambda$ on some side.
\end{enumerate}
We also set $v(\lambda)$ equal to the number of labels in $\lambda$ strictly higher than $a_\lambda$. So $v(\lambda)=0$ iff $\lambda$ satisfies (a).

Furthermore, we denote by {\bf P}$[k]$ for $0\le k\le\min(m,n)$ the property that there is a constant $C_{k}$, such that \eqref{eqthm} is true for all $\lambda\in\Lambda$ with $\sharp[\lambda]=k$. We will freely use the constant $C:=C_\fg$ from Lemma~\ref{bound2}.

\begin{lemma} 
\label{newlem1}
Assume that property {\bf P}$[k-1]$ holds and consider $\lambda\in\Lambda$ with $\sharp[\lambda]=k$.
\begin{enumerate}
\item If $\lambda$ satisfies $($a$)$, we have
$$p^j_\lambda\le (m+n)^2CC_{k-1}\sum_{l=0}^j (j-l)^{k-2}\le (m+n)^2CC_{k-1}j^{k-1}.$$
\item If $\lambda$ satisfies $($b$)$, denote the lowest label in $\lambda$ strictly higher than $a_\lambda$ by $b_\lambda$, set $d:=b_\lambda-a_\lambda-1>0$. Then we have 
$$p^j_\lambda\le (m+n)^2CC_{k-1}j^{k-1}+\sum_{i=1}^yp^{j-d}_{\lambda_{(i)}},$$
for some $y<m+n$ with $\lambda_{(i)}\in\Lambda$ satisfying $($c$)$, $v(\lambda_{(i)})=v(\lambda)$ and $\sharp[\lambda_{(i)}]=k$.
 \end{enumerate}
\end{lemma}
\begin{proof}
We consider $\lambda\in\Lambda$ with $\sharp [\lambda]=k$ and assume it satisfies either $($a$)$ or $($b$)$. Set $a:=a_\lambda$. We refer to the side with strictly most occurrences of $a$ as the big side and the other as the small side. If there is an equal number of $a$ on each side, we choose the big and small side randomly. Denote the number of $a$'s appearing on the big side by~$y$.

Fix one occurrence of $a$ on the small side. We create $\lambda'\in\Lambda$ by replacing that label by $a+1$. As $\lambda$ satisfies $(a)$ or $(b)$, by construction $\lambda'$ has degree of atypicality $k-1$. If the small side is the left-hand side we set $T=E_a$, otherwise $T=F_a$. Then $TM(\lambda')$ has a standard filtration of length $y+1$, where~$\lambda$ is the lowest weight appearing. The $y$ other highest weights, which we denote by $\{\lambda^{(i)}\,|\, i=1,\cdots, y\}$, are obtained from $\lambda$ by raising one of the occurrences of $a$ on the big side and the fixed occurrence of $a$ on the small side. Thus we can define a module $\widetilde M\in\cO$ by the short exact sequence
\begin{equation}\label{sesproof}0\to \widetilde M \to T M(\lambda')\to M(\lambda)\to 0.\end{equation}
By considering the long exact sequence obtained by applying $\oplus_{\nu}\Hom_{\cO}(-,L(\nu))$ to \eqref{sesproof} in combination with {\bf P}$[k-1]$ and Lemma~\ref{bound2}, we find 
\begin{equation}\label{neweqnew}p^j_\lambda\le \sum_{i=1}^yp^{j-1}_{\lambda^{(i)}}+(m+n)CC_{k-1} j^{k-2},\end{equation}
where~$m+n$ is the dimension of the tautological module for $\mathfrak{gl}(m|n)$.

If $\lambda$ satisfies (a), so do the $\lambda^{(i)}$. So we can apply the above procedure on each of the $\lambda^{(i)}$, with the added simplification that the highest label in $\lambda^{(i)}$ appears only once on each side. Hence the analogue of $y$ is equal to $1$ in the following steps. Applying this $j-1$ times and using the estimate $y\le m+n$ proves $(1)$.

Now assume that $\lambda$ satisfies (b) and recall the constant $d$ introduced in the statement of part (2) of the lemma. If $d=1$, then the $\lambda^{(i)}$ satisfy~$($c$)$ and the claim in part (2) follows immediately from equation~\eqref{neweqnew}. If $d>1$, then the $\lambda^{(i)}$ satisfy~$($b$)$ and we can apply the procedure again on them. Hence we can repeat the procedure $d$ times, where again only the first time we will need a constant~$y$ bigger than $1$. This yields
$$p^j_\lambda\le (m+n)^2CC_{k-1}\sum_{l=0}^{d-1} (j-l)^{k-2}+ \sum_{i=1}^yp^{j-d}_{\lambda_{(i)}}$$
with $\lambda_{(i)}\in\Lambda$ obtained from $\lambda$ by adding $d$ to our fixed occurrence of $a$ on the small side and adding $d$ to the $i$th occurrence of $a$ on the big side. By construction, the $\lambda_{(i)}$ satisfy~(c). This completes the proof of part $(2)$.
\end{proof}

\begin{lemma}
\label{newlem2}
Assume that $\lambda\in\Lambda$ satisfies $(c)$ and $\sharp[\lambda]=k$, then we have
$$p^j_\lambda\le (m+n)C \left(p_{\lambda'}^j+p_{\lambda''}^{j-1}\right),$$
with $\lambda',\lambda''\in\Lambda$ satisfying $\sharp[\lambda']=\sharp[\lambda'']=k$, $v(\lambda')< v(\lambda)$ and $v(\lambda'')<v(\lambda)$.
\end{lemma}
\begin{proof}
We define $\lambda'$ as obtained from $\lambda$ by raising the occurrence of $a_\lambda$ on the side where no $a_\lambda+1$ appears by $1$ and $\lambda''$ as obtained from $\lambda$ by raising both occurrences of $a_\lambda$ by one. By definition there is a short exact sequence
$$0\to M(\lambda'')\to T M(\lambda')\to M(\lambda)\to 0$$
with $T=E_a$ if $a_\lambda+1$ appears on the right and $T=F_a$ otherwise. The corresponding long exact sequence and Lemma~\ref{bound2} then imply
$$p^j_\lambda\le (m+n)C p_{\lambda'}^j+p_{\lambda''}^{j-1}.$$
The properties of $\lambda',\lambda''$ follow from construction.
\end{proof}

\begin{lemma}
\label{newlem3}
Assume that $\lambda\in\Lambda$ with $\sharp[\lambda]=k$ satisfies $(d)$. Then we have
$$p_\lambda^j\le (m+n)^{m+n}C^{m+n} p^j_{\lambda_{+}}, $$
for some $\lambda_+\in\Lambda$ satisfying $($b$)$, $\sharp[\lambda_+]=k$ and $v(\lambda_+)=v(\lambda)$.
\end{lemma}
\begin{proof}
We define $\lambda_+$ as obtained from $\lambda$ by adding $1$ to every label strictly bigger than~$a_\lambda$. By composing the appropriate $E_i$ and $F_i$ ($v(\lambda)$ in total) into a translation functor $T$ we have $M(\lambda)^{\oplus d}=T M(\lambda_+)$ for some number $d$ with $1\le d\le v(\lambda)!$. Lemma~\ref{bound2} then implies
$$p_\lambda^j\le (m+n)^{v(\lambda)}C^{v(\lambda)} p^j_{\lambda_{+}}, $$
which proves the lemma.
\end{proof}

\begin{lemma}
\label{lemcVerma}
Assume that {\bf P}$[k-1]$ holds, with $1\le k\le \min\{m,n\}$, then also {\bf P}$[k]$ holds.
\end{lemma}
\begin{proof}
We will prove by induction on $v\in [0,m+n-2k]$ that there is a constant $C^{(v)}_{k}$ such that if $\lambda$ is of atypicality degree $k$ and $v(\lambda)\le v$, then $p_\lambda^j\le C_{k}^{(v)} j^{k-1}$ holds for all $j>0$. This proves the lemma for $C_{k}=C_{k}^{(m+n-2k)}$.

If $v(\lambda)=0$, then $\lambda$ satisfies (a), so Lemma~\ref{newlem1}(1) implies this result with $C^{(0)}_{k}=(m+n)^2CC_{k-1}$. Now assume that the property holds for all $v$ up to $\widehat{v}$ and consider $\lambda\in\Lambda$ with $\sharp[\lambda]=k$ and $v(\lambda)=\widehat{v}+1$.

\begin{enumerate}[(i)]
\item If $\lambda$ satisfies (c), then the induction hypothesis and Lemma~\ref{newlem2} imply
$$p^j_\lambda\le D_1 j^{k-1}\qquad\mbox{with}\quad D_1:=2(m+n)CC^{(\widehat v)}_{k}.$$
\item If $\lambda$ satisfies (b), then (i) and Lemma~\ref{newlem1}(2) imply
$$p^j_\lambda\le (m+n)^2CC_{k-1}j^{k-1}+(m+n)D_1 (j-d)^{k-1}\le D_2 j^{k-1}$$
for $D_2:=(m+n)D_1+(m+n)^2CC_{k-1} $.
\item If $\lambda$ satisfies (d), then (ii) and Lemma~\ref{newlem3} imply
$$p^j_\lambda\le D_3 j^{k-1}\qquad\mbox{with}\quad D_3:=(m+n)^{m+n}C^{m+n}D_2.$$
\end{enumerate}
Thus we can take $C_{k}^{(\widehat v+1)}=D_3$, which concludes the proof.
\end{proof}

\begin{proof}[Proof of Theorem \ref{complVerma}]
First we note that there is a constant $C_{0}$, such that property {\bf P}[0] holds. This follows from the equivalence of this question to the one in blocks of $\cO_{\Z}^0$, see e.g. Lemma \ref{KLorbit}, since there are finitely many non-equivalent blocks each containing finitely many Verma modules. Lemma~\ref{lemcVerma} then iteratively proves the first statement in the theorem and thus also $c_{\cO}(M(\lambda))\le \sharp[\lambda]$.

Now we establish the equality for regular weights. First we take $\kappa\in\Lambda^{++}$ and use Lemma~\ref{lemOF} to obtain 
\begin{eqnarray*}
\sum_{\nu\in\Lambda^{++}}\dim\Ext_{\cO}^j(M(\kappa),L(\nu))=\sum_{\nu\in\Lambda^{++}}\Ext^j_{\cF}(K(\kappa), L(\nu)).
\end{eqnarray*}
This has polynomial growth rate $\sharp[\kappa]$ by Lemma~\ref{lemF}. Now for any $\kappa\in\Lambda^{++}$ and $w\in W$ we consider the subsequence of \eqref{KLVcompl}
$$\sum_{\nu\in\Lambda^{++}}\dim\Ext^j_{\cO}(M(w\kappa), L(\nu))= \sum_{\nu\in\Lambda^{++}}\dim\Ext^{j-l(w)}_{\cO}(M(\kappa), L(\nu)),$$
where the equality follows from Lemma~\ref{lemtwist}$(ii)$. This proves that $p^j_{w\kappa}$ has polynomial growth rate at least $\sharp [w\kappa]$.
\end{proof}

\begin{proof}[Proof of Theorem \ref{complK}]
First we prove that $c_{\cO}(K(\lambda))\le \sharp[\lambda]$ for any $\lambda\in \Lambda$. For an anti-dominant $\mu$, we have $M(\mu)=K(\mu)$, so the result follows from Theorem \ref{complVerma}. Then we use (finite) induction by considering the Bruhat order $\prec_0$ for $\mathfrak{g}_0$ on $P_0$. Assume that $c_{\cO}(K(\nu))\le k$ for all $\nu\prec_0\lambda$ with $\sharp[\lambda]=k$. Then the module $N$ defined by the exact sequence
$$0\to N\to M(\lambda)\to K(\lambda)\to 0,$$
has a filtration by $K(\nu)$ with $\nu\prec_0\lambda$. By the induction hypothesis and Lemma~\ref{lemcompseq} we have $c_{\cO}(N)\le \sharp[\lambda]$. Lemma~\ref{lemcompseq} and Theorem \ref{complVerma} then imply $c_{\cO}(K(\lambda))\le \sharp[\lambda]$.

Now take an arbitrary module which is $\fg_{-1}$-free. By Proposition \ref{propflag} this module has a filtration by Kac modules, the result thus follows from Lemma~\ref{lemcompseq}. All results are also valid for the anti-distinguished system of positive roots, which proves the claim for~$\fg_1$.
\end{proof}

\subsection{Complexity of simple modules for $\mathfrak{gl}(m|n)$} Also the complexity of simple modules is in principle determined by Brundan's KL polynomials, see Corollary \ref{corsim}. 

We investigate a relation between the complexity of a simple module and its $\fn$-cohomology. Therefore we introduce
$$c_{\fn}(M):=r(\dim H^\bullet(\fn,M)),$$
for $M\in\cO$, with $r$ as introduced in Definition \ref{defC}.
\begin{proposition}
\label{compKost}
For any $\lambda\in\Lambda$, we have
$$\max\{c_{\cO}(M(\lambda))\,,\,c_{\fn}(L(\lambda))\}\,\,\le\,\, c_{\cO}(L(\lambda))\,\,\le\,\,\sharp[\lambda]+c_{\fn}(L(\lambda)).$$
\end{proposition}
\begin{proof}
Set $\sharp[\lambda]=k$, by equation~\eqref{extsimple} and Theorem \ref{complVerma} we have
$$\sum_{\mu\in\Lambda}\Ext^j_{\cO}(L(\lambda),L(\mu))\le\sum_{i=0}^j C_{k}(j-i)^{k-1}\sum_{\kappa}\dim\Ext^i_{\cO}(M(\kappa),L(\lambda)),$$
By equation~\eqref{VerH} and by setting $p=c_{\fn}(L(\lambda))$, there exists some constant $C$ for which
$$\sum_{\kappa\in\Lambda}\dim \Ext^i_{\cO}(M(\kappa),L(\lambda))\le C i^{p-1},\quad\qquad\forall i\in\N.$$
Combining the two above equations leads to
$$\sum_{\mu\in\Lambda}\Ext^j_{\cO}(L(\lambda),L(\mu))\le\sum_{i=0}^j C_{k}(j-i)^{k-1}Ci^{p-1}\le C C_{k} j^{k+p-1},$$
which implies the second inequality. 

By considering only the extremal terms in the summation~\eqref{extsimple}, we find
$$\dim\Ext^j_{\cO}(L(\lambda),L(\mu))\;\ge \;\max\left\{\dim \Ext^j_{\cO}(M(\lambda),L(\mu)),\dim \Ext^j_{\cO}(M(\mu),L(\lambda)) \right\}.$$
The first inequality in the claim then follows from equation~\eqref{VerH}.
\end{proof}

For finite dimensional simple modules we can improve the estimates.
\begin{proposition}
If $\kappa\in\Lambda^{++}$, we have 
$$2\sharp[\kappa]\,\le\, c_{\cO}(L(\kappa))\,\le\,\sharp[\kappa]+r\left(\sum_{\nu\in\Lambda^{++}}\dim\Ext^\bullet_{\cF}(K(\nu),L(\kappa))\right).$$
\end{proposition}
\begin{proof}
Equation~\eqref{extsimple} gives the following lower bound for $\sum_{\mu\in\Lambda}\dim\Ext^j_{\cO}(L(\kappa),L(\mu))$:
\begin{equation}\label{subseq}\sum_{\lambda,\nu\in\Lambda^{++}}\sum_{i=0}^j\dim\Ext^{i}_{\cO}(M(\lambda),L(\kappa))\dim\Ext^{j-i}_{\cO}(M(\lambda),L(\nu)).\end{equation}
By Lemma~\ref{lemOF} and the abstract KL theory of $\cF$, see Theorem 4.51 and Corollary 4.52 in~\cite{Brundan}, we then find that the summation in \eqref{subseq} is equal to $\sum_{\nu\in\Lambda^{++}}\Ext^j_{\cF}(L(\kappa),L(\nu)).$
This has polynomial growth rate $2\sharp[\kappa]$ by Theorem 9.1.1 in~\cite{BKN2}, proving the first inequality.

By Lemma~\ref{lemtwist}$(i)$ and $(ii)$ and Lemma~\ref{lemOF} we have
$$\dim H^j(\fn,L(\kappa))=\sum_{i=0}^{l(w_0)}\left(\sharp W(i) \right)\sum_{\lambda\in\Lambda^{++}} \dim\Ext^{j-i}_{\cF}(K(\lambda),L(\kappa)),$$
with $\sharp W(i)$ the number of elements in $W$ of length $i$. This proves the second inequality by Proposition \ref{compKost}.
\end{proof}

We end this subsection with a conjecture.
\begin{conjecture}
\label{consimple}
For any $\lambda\in\Lambda$ we have $c_{\cO}(L(\lambda))=2\sharp[\lambda]$ and $c_{\cO}(M(\lambda))=\sharp[\lambda]$.
\end{conjecture}

If this conjecture is true we obtain in particular that for an integral block $\cO_{\xi}$
\begin{itemize}
\item a categorical interpretation of the singularity, by the finitistic global dimension $2\mathtt{a}(w_0w_0^{\xi})$, see Theorem \ref{fdblock}.
\item a categorical interpretation of the atypicality, by the global complexity $2\sharp\xi$, see Conjecture \ref{consimple}.
\end{itemize}

\subsection{Link between complexity and associated variety}

We note two explicit connections between complexity in category~$\cO$ for $\mathfrak{gl}(m|n)$ and the associated variety, which follow from Theorem \ref{finpdind}, Lemma~\ref{general}(2) and Theorem \ref{complK}.
\begin{proposition}
\begin{enumerate}
\item For any $M\in\cO$, we have $c_{\cO}(M)=0\Leftrightarrow X_M=\{0\}$.
\item If $M\in\cO$ is $\fg_{-1}$-free and admits a generalised central character of atypicality degree $k$, we have both $c_{\cO}(M)\le k$ and $|S|\le k$ for any $S\in \cS(M)$.
\end{enumerate}
\end{proposition}
This results seem to suggest that there must be some deeper connection between complexity and the associated variety. Similar connections appear in~\cite{BKN2}.

\appendix

\section{Example: $\mathfrak{gl}(2|1)$}
\label{gl21}
Lusztig's canonical basis for the principal block in category~$\cO$ for $\mathfrak{gl}(2|1)$ (with distinguished system of positive roots) has been explicitly calculated in Section 9.5 of~\cite{Cheng}, see also Example 3.4 in~\cite{Brundan3}. The principal block is the block containing the trivial module, which has highest weight $\mu^0=(-1-2|-2)$. Instead we consider the equivalent block containing $(10|1)$.

We apply the procedure in Subsection \ref{seclen}, to obtain the length function. Consider $a<<0$ and $b>>0$, then
\begin{eqnarray*}
\phi_{[a,b]}(10|1)&=&(1,0,b+1,b,\cdots ,2,0,-1,\cdots ,a)\\
\phi_{[a,b]}(01|1)&=&(0,1,b+1,b,\cdots ,2,0,-1,\cdots ,a)\\
\phi_{[a,b]}(00|0)&=&(0,0,b+1,b,\cdots,1,-1,-2,\cdots,a)\\
\phi_{[a,b]}(0-1|-1)&=&(0,-1,b+1,b,\cdots,0,-2,-3,\cdots,a),
\end{eqnarray*}
where similar expression hold for all other $\mu\in[(00,0)]$. All the weights on the right-hand side do indeed belong to the same Weyl group orbit. The corresponding longest elements of the Weyl group $S_{b+2-a}$ are given by
\begin{eqnarray*}
&& (s_{2}s_3\cdots s_{b+1})   (s_1 s_2 \cdots  s_{b})s_{b+2}\\
&& s_1 (s_{2}s_3\cdots s_{b+1})   (s_1 s_2 \cdots  s_{b})s_{b+2}=(s_2s_3\cdots s_{b+1})(s_1s_2\cdots s_{b+2})\\
&& (s_2s_3\cdots s_{b+1})s_{b+2}(s_1s_2\cdots s_{b+2})   \\
&&   (s_2s_3\cdots s_{b+1})s_{b+2}s_{b+3}(s_1s_2\cdots s_{b+2}) ,
\end{eqnarray*}
where we used the standard convention for the notation of simple reflections. This also confirms that the Bruhat order for $\mathfrak{gl}(m|n)$ is indeed translated to that on the Weyl group under $\phi_{[a,b]}$.
We observe that the difference in lengths between these elements does not depend on the precise choice of $b$ and $a$, confirming Lemma~\ref{defL}. In particular we find
$$\mathtt{l}((10|1),(01|1))=1=\mathtt{l}((01|1),(00,0))=\mathtt{l}((00|0),(0-1|-1)).$$
The same type of calculation quickly reveals that the Bruhat order is transitively generated by relations $\mu\preceq\lambda$ where~$\mathtt{l}(\lambda,\mu)=1$. In other words, all coverings in the Bruhat order correspond to length 1. On the left of Figure~\ref{fig2}, we denote the elements of $[(10|1)]$, along with arrows between them, representing the coverings in the Bruhat order. By transitivity, the length function for arbitrary related weights can be read of from Figure~\ref{fig2}. This provides an example of Remark \ref{lengthnotsame}(2), since $\mathtt{l}((10|1),(0-1|-1))=3$, while $l(10|1)-l(0-1|-1)=1$ for Brundan's length $l$ on $\Lambda^{++}$ in Section $3$-$g$ of~\cite{Brundan}.

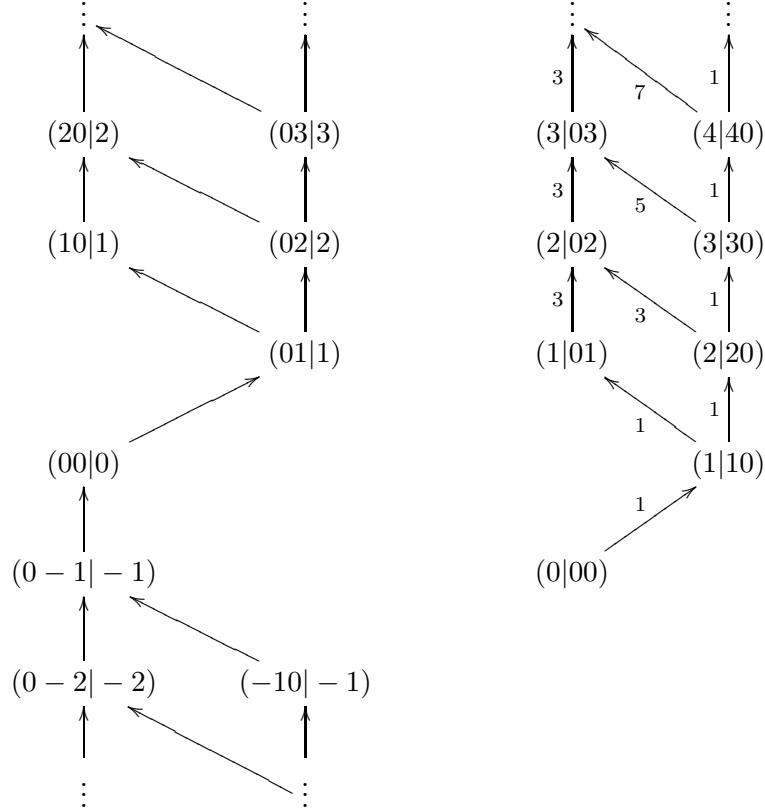
\begin{figure}\caption{The length function for $\mathfrak{gl}(2|1)$ and $\mathfrak{gl}(1|2)$}
\begin{displaymath}
    \xymatrix{ 
\vdots &\vdots &&\vdots &\vdots\\
(20|2)\ar[u] & (03|3)\ar[u]\ar[ul] &&(3|03)\ar[u]^3&(4|40)\ar[u]^1\ar[ul]^{7}\\
(10|1)\ar[u] & (02|2)\ar[u]\ar[ul]&&(2|02)\ar[u]^3&(3|30)\ar[u]^{1}\ar[ul]^5\\
& (01|1)\ar[u]\ar[ul]&& (1|01)\ar[u]^{3}&(2|20)\ar[u]^1\ar[ul]^3\\
(00|0)\ar[ur]& &&&(1|10)\ar[u]^1\ar[ul]^1 \\
(0-1|-1)\ar[u]& && (0|00)\ar[ur]^1\\
(0-2|-2)\ar[u]& (-10|-1)\ar[ul]\\
\vdots\ar[u] & \vdots\ar[u]\ar[ul]
   }
\end{displaymath}\label{fig2}
\end{figure}

Inverting the KL polynomials of Section 9.5 of~\cite{Cheng}, see also Example 3.4 in~\cite{Brundan3} show that the interesting KL polynomials for $\mathfrak{gl}(2|1)$, meaning around the singular point $(00|0)$, are given by
\begin{eqnarray*}
\dot{v}_{(10|1)}=\sum_{k=1}^\infty (-q)^{k-1} \dot{b}_{(k0|k)},&&\dot{v}_{(01|1)}=\sum_{k=1}^\infty (-q)^{k-1}\dot{b}_{(0k|k)}+\sum_{k=1}^\infty (-q)^k \dot{b}_{(k0|k)},\\
\dot{v}_{(00|0)}=\sum_{k=0}^\infty (-q)^k \dot{b}_{(0k|k)},&&\dot{v}_{(0-1|-1)}= \dot{b}_{(0-1|-1)}+\sum_{k=0}^\infty(-q)^{k+1} \dot{b}_{(0k|k)}+\sum_{k=1}^\infty(-q)^k \dot{b}_{(k0|k)},\\
&& \dot{v}_{(-10|-1)}=\dot{b}_{(-10|-1)}-q\dot{b}_{(0-1|-1)}+\sum_{k=1}^\infty (-q)^{k+1} \dot{b}_{(k0|k)}.
\end{eqnarray*}
In particular, these confirm Theorem \ref{KLtheory} and Lemmata \ref{vancentre} and \ref{lemtwist}.

The left column of the left graph in Figure \ref{fig2} gives the elements of $\Lambda^{++}$. The above KL polynomials also imply that the $\Ext^1$-quiver for the principal block in category~$\cO$ for $\mathfrak{gl}(2|1)$ is obtained by replacing all arrows in the left diagram by $\leftrightarrow$ and adding one $\leftrightarrow$ between $(10|1)$ and $(0-1|-1)$. In particular this implies the well-known property that the $\Ext^1$-quiver of the principal block of $\cF$ is of Dynkin type~$A_\infty$.

On the right of Figure~\ref{fig2} we give the corresponding information for the principal block of $\mathfrak{gl}(1|2)$ with distinguished Borel subalgebra. This case is naturally isomorphic to $\mathfrak{gl}(2|1)$ with anti-distinguished Borel subalgebra. The corresponding calculation for the length function now reveals the length function here does not lead to coverings with difference in length function equal to 1. The length function corresponding to the coverings is denoted on the arrows on the right in Figure~\ref{fig2}. The case $\mathfrak{gl}(1|2)$ clearly provides an example of Remark \ref{lengthnotsame}(1).

\section{Existence of upper bounds}

In this appendix a classical Lie superalgebra $\fg$ is always assumed to have even Cartan subalgebra $\fh$. We denote by $L_{\oa}(\nu)$ the simple $\fg_{\oa}$-module with highest weight $\nu\in\fh^\ast$.

\begin{lemma}
\label{bound2}
For any classical Lie superalgebra $\fg$, there exists a constant $C_{\fg}$, such that the number of indecomposable projective modules in 
$P(\lambda)\otimes V$ is bounded by $C_{\fg}\dim V$ for any $\lambda\in\fh^\ast$ and $\fg$-module~$V$.
\end{lemma}
\begin{proof}
We set $N_{\lambda,V}:=\sum_{\mu\in\fh^\ast}\dim\Hom_{\fg}(P(\lambda)\otimes V,L(\mu)).$ This is the number of times $L(\lambda)$ appears as a subquotient in the product $L(\mu)\otimes V^\ast$. This is smaller than the number of times $L_{\oa}(\lambda)$ appears in $\Res(M(\mu)\otimes V^\ast)$ as a subquotient. We denote $\Res V^\ast$ simply by~$V^\ast$ and calculate
\begin{eqnarray*}
N_{\lambda,V}&\le &\sum_{\mu\in\fh^\ast}[M_{\oa}(\mu)\otimes \Lambda\fg_{-1}\otimes V^\ast\,:\, L_{\oa}(\lambda)]\\
&=&\sum_{\mu,\kappa\in\fh^\ast}\dim \left(\Lambda\fg_{-1}\otimes V^\ast\right)_{\kappa}[M_{\oa}(\mu+\kappa)\,:\, L_{\oa}(\lambda)].
\end{eqnarray*}
The terms in the sum on the right-hand side are zero unless $\lambda$ is in the Weyl group orbit of $\mu+\kappa$, so 
\begin{eqnarray*}
N_{\lambda,V}&\le&\sum_{w\in W}\sum_{\mu\in\fh^\ast}\dim \left(\Lambda\fg_{-1}\otimes V^\ast\right)_{w\circ \lambda-\mu}[M_{\oa}(w\circ\lambda)\,:\, L_{\oa}(\lambda)]\\
&\le & \sharp W\,d\,\dim(\Lambda\fg_{-1})\,\dim V
\end{eqnarray*}
with $\sharp W$ the order of the Weyl group and $d$ the maximal length of a $\fg_{\oa}$-Verma module.
\end{proof}

\begin{lemma}
\label{bound3}
For any classical Lie superalgebra $\fg$, there exists a constant $\widetilde{C}_{\fg}$, such that for an arbitrary~$\nu\in\fh^\ast$ we have
$$\sum_{\mu\in\fh^\ast}[\Res L(\mu): L_{\oa}(\nu)]\le\widetilde{C}_{\fg}.$$
\end{lemma}
\begin{proof}
By Frobenius reciprocity, this sum is equal to 
$$\sum_{\mu\in\fh^\ast}\Hom_{\fg}(\Ind P_{\oa}(\nu), L(\mu)),$$
with $P_{\oa}(\nu)$ the projective cover of $L_{\oa}(\nu)$ in $\cO^0$. This sum is the number of indecomposable projective modules in $\cO$ in the decomposition of $\Ind P_{\oa}(\nu)$. The sim is hence smaller than the number of indecomposable projective modules in $\cO^0$ in the decomposition of \begin{equation}\label{resindP}\Res\Ind P_{\oa}(\nu)\simeq \Lambda\fg_{\ob}\otimes P_{\oa}(\nu).\end{equation}
The result therefore follows from Lemma~\ref{bound2} applied to $\fg_{\oa}$, with $\widetilde{C}_{\fg}=C_{\fg_{\oa}}\dim\Lambda\fg_{\ob}$.
\end{proof}

\subsection*{Acknowledgment}

KC is a Postdoctoral Fellow of the Research Foundation - Flanders (FWO). VS was partially supported by NSF grant
1303301. The authors would like to thank Volodymyr Mazorchuk and Jonathan Brundan for helpful discussions and
the referee for several useful comments.

\end{document}